\documentclass[11pt]{article}
\usepackage{amssymb,amsmath,amsfonts,amsthm,mathtools,color}

\usepackage{mathrsfs}
\usepackage{dsfont}
 \usepackage{bm} 

\usepackage{url} 

\usepackage[overload]{empheq}

\usepackage{comment} 
\usepackage[title,titletoc,header]{appendix}
\usepackage{graphicx,subfig}
\usepackage{float} 
\usepackage{dsfont} 

\usepackage{paralist}

\usepackage{cases}

\usepackage{tikz}
\usetikzlibrary{arrows,positioning,shapes.geometric}

\usepackage[linesnumbered, ruled,vlined]{algorithm2e}
\SetKwInput{KwInput}{Input}                

\graphicspath{ {./figures/} }

\usepackage{indentfirst}
\usepackage{multicol}
\usepackage{booktabs}
\usepackage{url}
\usepackage[outdir=./]{epstopdf} 

\usepackage[shortlabels]{enumitem}

\usepackage{hyperref}
\hypersetup{
    colorlinks=true, 
    linktoc=all,     
    linkcolor=blue,  
}
\numberwithin{equation}{section}

\newtheorem{Theorem}{Theorem}[section]
\newtheorem{Lemma}[Theorem]{Lemma}
\newtheorem{Proposition}[Theorem]{Proposition}
\newtheorem{Corollary}[Theorem]{Corollary}
\newtheorem{Assumption}{H.\!\!}

\theoremstyle{definition}
\newtheorem{Definition}{Definition}[section]

\newtheorem{Example}{Example}[section]

\theoremstyle{remark}
\newtheorem{Remark}{Remark}[section]

\def\mapto{\rightarrow}

\def\eps{\varepsilon}

\def\ms{\medskip}

\def\cA{\mathcal{A}}
\def\cB{\mathcal{B}}

\def\cD{\mathcal{D}}

\def\cF{\mathcal{F}}

\def\cL{\mathcal{L}}
\def\cM{\mathcal{M}}

\def\cP{\mathcal{P}}

\def\cU{\mathcal{U}}

\def\cW{\mathcal{W}}
\def\cX{\mathcal{X}}
\def\cY{\mathcal{Y}}
\def\cZ{\mathcal{Z}}

\def\d{{\mathrm{d}}}

\def\sB{\mathbb{B}}

\def\sE{{\mathbb{E}}}
\def\sF{{\mathbb{F}}}

\def\sL{{\mathbb{L}}}

\def\sN{{\mathbb{N}}}
\def\sP{\mathbb{P}}

\def\sR{{\mathbb R}}
\def\sS{{\mathbb{S}}}
\def\sT{{\mathbb T}}

\newcommand{\tr}{\textnormal{tr}}

\DeclareMathOperator*{\argmax}{arg\,max}
\DeclareMathOperator*{\argmin}{arg\,min}

\newcommand{\lc}
{\mathrel{\raise2pt\hbox{${\mathop<\limits_{\raise1pt\hbox
{\mbox{$\sim$}}}}$}}}

\newcommand{\gc}
{\mathrel{\raise2pt\hbox{${\mathop>\limits_{\raise1pt\hbox{\mbox{$\sim$}}}}$}}}

\newcommand{\ec}
{\mathrel{\raise2pt\hbox{${\mathop=\limits_{\raise1pt\hbox{\mbox{$\sim$}}}}$}}}

\def\bb{\begin{equation}} \def\ee{\end{equation}}
\def\bbn{\begin{equation*}} \def\een{\end{equation*}}

\def\beqn{\begin{eqnarray}}  \def\eqn{\end{eqnarray}}

\def\beqnx{\begin{eqnarray*}} \def\eqnx{\end{eqnarray*}}

\def\bn{\begin{enumerate}} \def\en{\end{enumerate}}

\def\bd{\begin{description}} \def\ed{\end{description}}



\begin{document}

\title{Continuous-time mean field games: a primal-dual characterization} 

\author{
Xin Guo \thanks{University of California, Berkeley, Department of IEOR, email: xinguo@berkeley.edu}
\and
Anran Hu \thanks{Columbia University, Department of IEOR, email: ah4277@columbia.edu}
\and
Jiacheng Zhang \thanks{The Chinese University of Hong Kong, Department of Statistics, email: jiachengzhang@cuhk.edu.hk}
\and
Yufei Zhang \thanks{Imperial College London, Department of Mathematics, email: yufei.zhang@imperial.ac.uk}
}

\date{}
\maketitle

\begin{abstract}
This paper establishes a primal-dual formulation for   continuous-time mean field games (MFGs) and provides a complete analytical characterization for  the set of all mean field  equilibria (MFEs).
 We reformulate the representative player's control problem as a linear program over occupation measures and derive its dual formulation over smooth subsolutions of the associated Hamilton--Jacobi--Bellman (HJB) equation.  This primal-dual formulation yields a verification theorem for MFEs under measurable coefficients, and provides necessary and sufficient conditions characterizing \emph{all MFEs} under a strong duality condition. We establish this strong duality for nondegenerate diffusions with H\"older continuous coefficients through a new global $C^{2,\alpha}$ regularity result for fully nonlinear  parabolic HJB equations.  

A key insight  is that the dual variable in the primal-dual framework need only coincide with the HJB solution on the support of the mean field flow.
Compared with existing approaches, our characterization requires neither convexity of the Hamiltonian nor uniqueness of its optimizer, accommodates relaxed equilibria, and remains applicable when the HJB equation lacks classical solutions.

\end{abstract}

\medskip
\noindent
\textbf{Key words.} 
Mean field game, 
Nash equilibrium,
primal-dual,
occupation measure,
controlled martingale, 
superposition principle, 
strong duality

\ms
\noindent
\textbf{AMS subject classifications.} 
91A16, 90C46, 49L12

\medskip

 \section{Introduction}

\paragraph{MFGs and their equilibria.}
The theory of mean field games (MFGs) was pioneered in the seminal works of \cite{huang2006large} and \cite{lasry2007mean}. 
Its ingenious idea of assuming a population of homogeneous players with weak interactions has reduced the analysis of  nonzero-sum games  to finding the optimal strategy of a single representative player  which is consistent with the population distribution. Moreover,  an  equilibrium  in the MFG has been shown to serve as an $\epsilon$-Nash equilibrium for the corresponding finite-player game. As such, MFG theory provides an innovative framework for approximating Nash equilibria in finite-player dynamic games that would otherwise be intractable. 

A fundamental question in MFGs is to characterize  mean field equilibria (MFEs). An MFE is defined by two key criteria:
the first is the optimality condition for the representative player, which ensures that the chosen strategy is optimal given the mean field distribution of the entire population; the second is the consistency condition, which requires that the state dynamics induced by the representative player's strategy match the distribution of the population state. 

To characterize the MFEs, there are two main analytical approaches: a partial differential equation (PDE) approach and a probabilistic approach. The PDE approach is based on the coupled HJB-FP system: a backward Hamilton-Jacobi-Bellman (HJB) equation that determines the optimal value function of the representative player  given the population distribution, and a forward Fokker-Planck (FP)  equation  that describes the evolution of the population distribution under the optimal feedback control derived from the HJB equation (see, e.g., \cite{huang2006large,lasry2007mean,bayraktar2019rate}). A closely related PDE formulation is provided by the master equation, which involves deriving and analyzing  a PDE on the space of probability measures based on the coupled forward-backward system \cite{bensoussan2015master,cardaliaguet2019master,bayraktar2018analysis}. The probabilistic approach  replaces the  HJB equation with a backward stochastic differential equation (BSDE) \cite{buckdahn2009mean, carmona2021probabilistic, dianetti2021submodular}, offering a probabilistic approach  to characterize the optimal control of a representative player.

Despite their success, these classical approaches often  rely explicitly or implicitly on strong assumptions such as convexity of the Hamiltonian or uniqueness of its optimizer.
These conditions ensure that the representative player admits a unique optimal control given the population distribution. However, in general, a representative player may have multiple optimal controls, and MFGs may admit multiple equilibria, {including   relaxed equilibria induced by relaxed controls \cite{lacker2015mean}}.
Uniqueness of MFEs is more the exception than the rule.
This has led to increasing interest in developing methods that can characterize the full set of equilibria rather than a single solution (see e.g., \cite{bardi2019non, delarue2020selection, alekos2026convergence}).

\paragraph{Characterization of the MFE set  and key challenges.}

Recently, a primal-dual formulation has  been  developed in \cite{guo2024mf} for discrete-time-and-state MFGs  using the complementarity condition in a linear program. By introducing dual variables, this approach manages to characterize the  set of all MFEs for {\it discrete-time} MFGs. A natural question is whether such a primal-dual formulation can be established for continuous-time and continuous-state  MFGs.  

  Several challenges arise in developing a primal–dual formulation for continuous-time and continuous-state mean field games (MFGs). The first is to show that, for any given mean field flow, the representative player’s control problem with measurable coefficients is equivalent to a linear program over the space of occupation measures. In contrast to existing works \cite{dumitrescu2021control, dumitrescu2023linear, cohen2024existence}, which assume continuous coefficients, this generalization is essential for characterizing MFEs, where the optimal policy of the representative player is typically only measurable in the state variable.

The second, and primary, challenge is to establish the dual formulation of this linear program. The solution to the dual problem, for the purpose of characterizing MFEs, plays a role analogous to that of adjoint variables in identifying all optimal controls in classical control problems.

A further challenge is to establish strong duality between the primal and dual problems in order to characterize all MFEs. This requires identifying conditions under which the dual problem attains its maximizer, also known as the solvability of the dual problem. To the best of our knowledge, this issue has not been studied in the control literature.

 \paragraph{Our work.}

This paper develops a primal-dual formulation for   continuous-time and continuous-state MFGs,  provides a verification theorem for MFEs,  with a complete characterization of all MFEs under additional conditions. Neither result requires convexity of the Hamiltonian or uniqueness of its optimizer.

 The  first  result is to show   that fixing a  mean field flow, the representative player's control problem with  {\it measurable coefficients} is equivalent to  a linear program over the space of occupation measures
 (Theorem \ref{thm:primal_diffusion}). 
This equivalence is established via  an equivalent auxiliary  controlled martingale problem. The main technical effort in the analysis is  to show that, using the so-called superposition principle (Proposition \ref{prop:measure_diffusion}),  any appropriate  admissible occupation measure for the  linear program induces a weak solution of the controlled state process.   

The second 
result is to establish the dual formulation of this linear program 
as a maximization problem over smooth subsolutions of the associated HJB equation,
inspired by the analysis in \cite{lions1983hamilton, fleming1989convex, buckdahn2011stochastic, serrano2015lp} for classical control problems. 
Using the dual formulation, we establish a new verification theorem providing sufficient conditions for an MFE: for any given admissible occupation measure of the linear program, if there exists a solution to the dual problem   that achieves the same value, then the occupation measure corresponds to an MFE (Theorem \ref{thm:primal_dual_NE_diffusion_sufficient}). 

 Finally, a complete characterization of \emph{all MFEs} is  established,  under the additional condition  of  strong duality  between the linear program and its dual problem
(Theorem \ref{thm:primal_dual_NE_diffusion_necessary}).  
By analyzing the solvability of the dual problem,  we show that this strong duality condition holds, in particular, when    the fixed-flow  HJB equation admits a classical solution (Theorem \ref{thm:strong_duality_diffusion}).
We further derive explicit conditions on the model coefficients, showing that the primal-dual formulation provides both \emph{necessary and sufficient conditions for all MFEs} when the coefficients are H\"older continuous and the diffusion is nondegenerate
 (Proposition  \ref{prop:pde_regularity}). 
An essential technical step is to establish a new global
$C^{2,\alpha}$ regularity result for fully nonlinear uniformly
parabolic HJB equations with H\"older continuous
coefficients  (Theorem \ref{thm:evans-krylov}), 
relaxing both the Lipschitz continuity and bounded-domain assumptions  in the classical Evans--Krylov theory (Remark \ref{rmk:holder_measure}). 

We provide explicit examples highlighting the greater flexibility of the primal-dual formulation in characterizing MFEs relative to existing approaches, such as the classical HJB-FP formulation. Two features are particularly important.

First, the primal-dual formulation requires neither convexity of the Hamiltonian nor uniqueness of its optimizer, and can therefore capture equilibria induced by relaxed controls.  In particular, when the Hamiltonian admits multiple optimizers, an MFG may admit MFEs that cannot be realized by any strict feedback control.  The classical HJB-FP approach selects a Hamiltonian optimizer in advance  and may therefore   miss such equilibria. By contrast,  the primal-dual approach avoids this prior selection and  captures the equilibrium directly through an occupation measure   (see Example \ref{ex:discrete_action}).

Second, the dual variable need only coincide with the HJB solution on the support of the equilibrium mean field flow (Remark \ref{remark:distinction}). This localized requirement is 
analogous in spirit to the stochastic maximum principle, where adjoint processes and their decoupling fields are needed only along the optimal state trajectory. Consequently, when the equilibrium flow does not have full support, the dual problem may remain solvable even if the HJB equation admits no   classical solution, allowing the primal-dual formulation to characterize MFEs beyond the scope of classical HJB-FP  methods (see Example \ref{ex:nonsmooth_value}).

\paragraph{Most related works.}
Earlier work  of \cite{guo2024mf} studies primal-dual formulation for discrete-time MFGs with finite states and actions. 
In contrast, this paper focuses on the substantially  more challenging setting of  continuous-time, continuous-state MFGs, and builds on entirely different techniques. In particular,  
 the LP formulation is derived via an auxiliary controlled martingale problem by exploiting the superposition principle of diffusions;
the dual problem is formulated as a maximization problem over smooth subsolutions of the associated HJB equation; and  its solvability is established by analyzing the regularity of the HJB solution.

Linear programming formulation for the representative player's control problem in MFGs has been explored   in \cite{dumitrescu2021control, dumitrescu2023linear,cohen2024existence}.
However, as noted earlier, the linear program  {\it alone} is  insufficient to  characterize the   set of   MFEs. Its dual formulation, along with the strong duality between the primal and the dual problems established in this paper, is essential  in closing this gap.

{The connection  between controlled diffusion problems with continuous coefficients and controlled martingale problems is  known in the  literature. It has been used in \cite{nicole1987compactification, haussmann1990existence} for the  existence of optimal controls,
 in  \cite{lacker2015mean}  for the existence of MFEs, and  recently 
 in \cite{lacker2016general,djete2023large} to characterize MFEs  as limits  of $N$-player games. 
Our work extends    this connection   to MFGs with {\it measurable coefficients}, and moreover links the martingale problem framework   with LPs over occupation measures. Combined with the study of the dual problem, this facilitates a comprehensive characterization of both the equilibrium flows and equilibrium policies in MFGs
(Theorems \ref{thm:primal_dual_NE_diffusion_sufficient} and \ref{thm:primal_dual_NE_diffusion_necessary}). 

{Finally, the interpretation of the dual problem as the supremum over all smooth subsolutions of the associated HJB equation is classical \cite{lions1983hamilton, fleming1989convex, buckdahn2011stochastic, serrano2015lp}.
In settings with {\it continuous coefficients}, the dual problem  serves as an intermediate step linking the value function with the viscosity solutions of the HJB equation.  In our analysis with   measurable coefficients,
the dual problem 
{\it } plays an essential role in characterizing  MFEs.} 

\paragraph{Notation.}

Let $I_d$ denote  the $d\times d$ identity matrix, and $\sS^d$
denote  the space of symmetric $d\times d$ matrices, equipped with the
partial order $A\preceq B$ if $B-A$ is positive semidefinite.

  For any  topological space $X$, 
  we equip $X$ with the Borel $\sigma$-algebra $\cB(X)$.
  We denote by  $B_b(X)$   the   space  of   bounded   measurable functions on $X$,
  and  $C_b(X)$ the space of bounded continuous functions on $X$.
We denote by $\cM(X)$  the   space of finite signed measures on $X$,
   $\cM_+(X)$ the    positive cone 
  of finite (nonnegative) measures on   $X$,
    and   $\cP(X)$ 
  the space of probability measures on  $X$.
  We endow $\cP(X)$  with the topology of weak convergence. For any Polish spaces $X$ and $Y$, 
  we denote by $\cP(Y|X)$ the space of probability kernels 
  $\gamma: X\to \cP(Y)$.
We denote by 
$C^{1,2}_b([0,T]\times \sR^d)$
   the space of 
  continuous and bounded functions $
  u:[0,T]\times \sR^d\mapto \sR$  such that $\partial_t u$, 
  $D_x u$,
  and $D_x^2 u$ exist, and are continuous and bounded.
  We denote by $C^{2}_b( \sR^d)$ the space of time-independent functions in $C^{1,2}_b([0,T]\times \sR^d)$.

Proofs of the main results are deferred to Section~\ref{sec:proofs}.

\section{Mathematical setup of MFGs}
\label{sec:mfg}
We start with the following mathematical framework of a continuous-time  MFG, where
 there is a continuum of homogeneous rational players; see e.g., \cite{carmona2018probabilistic}. 
 In this game, each player controls a  continuous-time   diffusion process, to optimize a given objective function over a finite time horizon  $T<\infty$. Here both the drift and diffusion coefficients are controlled and depend on a given mean field state distribution. 
   An MFE  is achieved when the distribution of the optimal controlled state process coincides with the given mean field distribution. 

   \paragraph{Representative player's problem.}

  Let 
  $\sT =[0,T]$ with $T>0$ be the time horizon,
    $A$ be a metric space representing 
the   player's   action space, 
  $b:\sT\times 
  \sR^d\times A\times \cP(\sR^d)\mapto \sR^d$ 
  and   $\sigma:\sT\times \sR^d\times A\times \cP(\sR^d)\mapto \sR^{d\times d}$     
  be   the state dynamics' coefficients,
  and 
     $f:\sT \times \sR^d\times A\times \cP(\sR^d)\mapto \sR$ and 
     $g:  \sR^d\times   \cP(\sR^d)\mapto \sR$
     be the cost functions,
     satisfying 
the following conditions. 

\begin{Assumption}
\label{assm:diffusion_state}
  $A$ is a Polish space (i.e., complete separable metric space), 
  $b:\sT\times 
  \sR^d\times A\times \cP(\sR^d)\mapto \sR^d$, 
  $\sigma:\sT\times \sR^d\times A\times \cP(\sR^d)\mapto \sR^{d\times d}$,
   $f:\sT \times \sR^d\times A\times \cP(\sR^d)\mapto \sR$, and 
     $g:  \sR^d\times   \cP(\sR^d)\mapto \sR$
  are bounded and measurable.
\end{Assumption}

Given the   coefficients $b$ and $\sigma$,  
the state dynamics of the representative player follow a diffusion process with
a given initial distribution,
is 
controlled by   a relaxed policy, and 
evolves 
depending on the mean field distribution of the population's state.
More precisely,
let $\bm \mu =(\mu_t)_{t\in \sT } \in \cP(\sR^d|\sT )$
be a given mean field distribution of the population's state,
 $\rho\in \cP(\sR^d)$ be the distribution of the initial state,
and 
 $\cP(A| \sT\times \sR^d)$
be the collection of the representative player's policies.
Given 
 an 
$\sR^d$-valued   process 
$\bm X=(X_t)_{t\in \sT}$  on some   filtered probability space
$(\Omega, \cF, \sF=(\cF_t)_{t\in \sT}, \sP)$,
and a policy $\gamma \in \cP(A|\sT\times \sR^d) $, we call $(\bm X, \gamma) $
a  closed-loop  state-policy pair
if 
$\mathscr{L}^\sP(X_0)=\rho$
and 
$\bm X$
is a weak solution to 
the dynamics
\begin{equation}
\label{eq:state_diffusion}
\d X_t =b^{\bm \mu, \gamma} (t,X_t) \d t
+ \sigma^{\bm \mu, \gamma} (t,X_t) \d W_t,
\end{equation}
where $W$ is a $d$-dimensional $\sF$-Brownian motion on the probability space,
and 
$b^{\bm \mu, \gamma}:\sT\times \sR^d\mapto \sR^d $
and 
$\sigma^{\bm \mu, \gamma}:\sT\times \sR^d\mapto \sR^{d\times d}$
are measurable functions such that  
\begin{equation}\label{eq:def_bsigmamu}
b^{\bm \mu, \gamma}(t,x)
\coloneqq \int_A
b(t,x,a,\mu_t) \gamma(\d a|t,x),
\quad 
\sigma^{\bm \mu, \gamma}(t,x)
\coloneqq \sqrt{\int_A
(\sigma\sigma^\top) (t,x,a,\mu_t) \gamma(\d a|t,x)}
\end{equation}
with $\sqrt{M}$
being the principal square root of a positive  semidefinite matrix $M\in \mathbb{R}^{d\times d}$.

We denote by $\cA_{\rm cl}(\bm \mu)$ the set of all  closed-loop state-policy pairs $(\bm X, \gamma)$, with the given  mean field flow $\bm \mu$.
Observe that 
  the set 
$\cA_{\rm cl}(\bm \mu)$ may be 
 empty, 
as   (H.\ref{assm:diffusion_state}) only requires
the functions $b$ and $\sigma$
to be measurable in the state variable.
If the coefficients are more regular or 
the diffusion coefficient is  non-degenerate, 
then
$\cA_{\rm cl}(\bm \mu)$ is   non-empty;
 see (H.\ref{assum:duality}) in Section \ref{sec: strongdual}.

Given the set 
$\cA_{\rm cl}(\bm \mu) $
of closed-loop state-policy pairs,
the representative player considers the following minimization problem:
\begin{equation}\label{controlled_cl_diffusion}
\tag{CL}
\inf_{(\bm X,\gamma)\in \cA_{\rm cl}(\bm \mu) } J^{\bm \mu}_{\rm cl}(\bm X,\gamma),
\end{equation}
where $J^{\bm \mu}_{\rm cl}(\bm X,\gamma)$ is defined 
by
\begin{equation}\label{controlled_cl_loss}
J^{\bm \mu}_{\rm cl}(\bm X,\gamma)\coloneqq \sE^\sP \left[\int_{\sT}\int_A  f(t,X_t,a,\mu_t)\gamma(\d a|t,X_t)\,\d t+g(X_T,\mu_T) \right]
 \end{equation} 
 for   any  
  state-policy pair $(\bm X,\gamma)$ 
on     $(\Omega, \cF,  \sP)$. 
Under  Condition (H.\ref{assm:diffusion_state}),
 $\inf_{(\bm X,\gamma)\in \cA_{\rm cl}(\bm \mu) } J^{\bm \mu}_{\rm cl}(\bm X,\gamma)<\infty$
 if and only if $\cA_{\rm cl}(\bm \mu)\not=\emptyset$.

   \paragraph{MFG and its equilibria.}

In this MFG, 
the representative player optimizes \eqref{controlled_cl_loss}
by considering the  mean field flow $\bm \mu\in \cP(\sR^d|\sT )$ as given.
An equilibrium  is then achieved when 
the state process of an optimal  state-policy pair  for \eqref{controlled_cl_loss}
has the distribution $\bm \mu$. 

\begin{Definition}
\label{def:NE}
Let $\bm \mu^*\in \cP(\sR^d|\sT)$ and $(\bm X^*, \gamma^*)\in \cA_{\rm cl}(\bm \mu^*)$ be a state-policy pair
    on a probability space
$(\Omega, \cF,   \sP)$.
We say the tuple    $( \bm \mu^*, \bm X^*, \gamma^*) $  is a mean field equilibrium (MFE) 
if 
\begin{enumerate}[(1)]
\item
\label{item:optimality}
 $(\bm X^*,\gamma^*)$ is optimal  when  $\bm \mu^*$ is the given mean field  flow, i.e., 
$J^{\bm \mu^*}_{\rm cl}(\bm X^*,\gamma^*)\le J^{\bm \mu^*}_{\rm cl}(\bm X,\gamma)$
for all $(\bm X, \gamma)\in \cA_{\rm cl}(\bm \mu^*)$.

\item
\label{item:consistency}
$(\bm\mu^*, \bm X^*,\gamma^*)$ satisfies the consistency condition, i.e., 
$\mu^*_t =\mathscr{L}^\sP (X^*_t)$ for all $t\in \sT$.
\end{enumerate}
\end{Definition}

Definition \ref{def:NE} concerns the  closed-loop MFE,
in the sense that for a given $\bm \mu^*$, 
the representative player optimizes the cost functional 
\eqref{controlled_cl_loss}
over all Markov policies $\gamma\in \cP(A|\sT\times \sR^d)$. 
We formulate the MFG using relaxed policies taking values in \(\mathcal P(A)\). This framework includes MFGs with strict policies taking values in \(A\) as a special case and also accommodates equilibria induced by relaxed controls that cannot be realized by any strict   control (see Example \ref{ex:discrete_action} and \cite{lacker2015mean}).

\paragraph{Set of MFEs.}
Under Condition  (H.\ref{assm:diffusion_state}),   the  MFG may   have no MFE or multiple MFEs. 
The focus of our analysis is to characterize the set of MFEs given by Definition \ref{def:NE}. 
This is achieved by characterizing   the representative player's optimal controls
via a primal-dual approach, 
where the primal problem is a minimization problem over 
the space of     measures on  $\sT\times \sR^d\times A$,
and the dual problem is a  maximization problem over
  suitable   functions.

  \section{Primal formulation of     representative player's problem}
\label{sec:primal_diffusion}

To characterize the set of MFEs  
  in Section \ref{sec:mfg} via the primal-dual formulation,
the first step 
is to show   that for any given mean field flow 
$\bm \mu\in \cP(\sR^d |\sT)$, the representative player's control problem \eqref{controlled_cl_diffusion} is equivalent to  a linear program (LP)  over the space of     measures.
This  is established  by showing their  equivalence to an auxiliary     martingale problem. 
Such a  linear program  will be referred to   as  the primal problem. 

 \paragraph{The LP.}

 Given the control problem of the representative player, consider   two variables $(\nu,\xi)\in \cX_+\coloneqq \cM_+(\sR^d)\times  \cM_{+}(\sT\times \sR^d \times A)$, 
with 
 $\nu\in \cM_+(\sR^d)$ 
  the distribution of the state  at the terminal time $T$, and 
 $\xi$  the occupation measure associated with the
 control problem
(see e.g., \cite{stockbridge1990time,kurtz1998existence}), with   $\xi \in \cM_{+}(\sT\times \sR^d \times A)$ representing the state-action   distribution over the   time horizon $\sT$.

A pair of measures $(\nu,\xi)$ is in 
the admissible set $\mathcal{D}_{{P}}(\bm \mu)$ of the linear program if  
the state-time marginal  measure of $\xi$ 
coincides with 
the law of a controlled state process,
  and also is consistent with $\nu$   at the terminal time.
Specifically,
define the following linear  constraint  for  $(\nu,\xi)\in\cX_+$ such that
 \begin{align}\label{eq:condition_primal_cx_diffusion}
\begin{split}
&    \int_{\sR^d}  \psi(T,x)\nu(\d x)  
 - \int_{\sR^d} \psi(0,x)\rho(\d x)
\\
 &\quad  =      \int_{\sT \times \sR^d \times A}\Big(\big(\sL^{\bm \mu}  \psi\big)(t,x,a)+(\partial_t\psi)(t,x)\Big)\xi(\d t, \d x,  \d a),
    \quad \forall \psi \in \cW\coloneq C^{1,2}_b(\sT\times \sR^d),
    \end{split}
\end{align}
where the operator 
 $\sL^{\bm \mu}$ is given by
\begin{equation}
\label{eq:generator_diffusion}
(\sL^{\bm \mu} \psi)(t, x, a) =
\frac{1}{2}\textrm{tr}\big((\sigma\sigma^\top) (t,x,a, \mu_t) D^2_x\psi (t,x)\big)
+b(t,x,a, \mu_t)^\top 
D_x\psi (t,x). 
\end{equation}
Then the admissible set  $ \mathcal{D}_{{P}}(\bm \mu)\subset \cX_+
$ is defined by
\begin{equation} 
\label{feasible_set_diffusion}
 \mathcal{D}_{{P}}(\bm \mu)
 \coloneqq 
\left\{ 
(\nu,\xi)\in \cX_+ \, \big\vert \,  (\nu,\xi) \text{ satisfies } \eqref{eq:condition_primal_cx_diffusion}
\right\}.
\end{equation}

Given the  set $\mathcal{D}_{{P}}(\bm \mu)$ in \eqref{feasible_set_diffusion},
the  representative player considers the following   LP: 
\begin{equation}
\label{Linear_programming_diffusion}
\tag{LP}
   \inf_{(\nu,\xi)\in \mathcal{D}_P(\bm \mu) }J^{\bm \mu}_P(\nu,\xi),
\end{equation}
where
for all $(\nu,\xi)\in\cX_+$,
\begin{equation}
\label{eq:object_primal_cX_diffusion}
J_{{P}}^{\bm \mu} (\nu,\xi)\coloneqq  \int_{\sT \times \sR^d\times A}f(t,x,a,\mu_t)\xi (\d t, \d x,  \d a)+\int_{\sR^d }g(x,\mu_T)\nu(\d x),
\end{equation}
with the same cost functions $f$  and $g$ as in \eqref{controlled_cl_diffusion}.

Note that   the 
\eqref{Linear_programming_diffusion} is 
 over  measures 
 $\nu\in \cM_+(\sR^d)$ and $\xi \in \cM_{+}(\sT\times \sR^d \times A)$.
This is different from the linear programming formulation in  
\cite{dumitrescu2021control},
which replaces  $\xi$ by 
 a flow of finite measures 
$  t\mapsto m_t(\d x, \d a)$ that  is   integrable with respect to the Lebesgue measure over $\sT$.
Our   formulation 
\eqref{Linear_programming_diffusion}
has the advantage 
  that its dual problem
    is easier to characterize, which subsequently allows for characterizing the set of MFEs. This will be clear from our subsequent analysis.

\paragraph{Equivalence between 
  \eqref{controlled_cl_diffusion}   and \eqref{Linear_programming_diffusion}.}

In order to establish the equivalence of \eqref{controlled_cl_diffusion}   and \eqref{Linear_programming_diffusion}, we first  connect the feasible sets 
$\cA_{\rm cl}(\bm \mu)$ and 
$\mathcal{D}_{{P}}(\bm \mu)$ via an auxiliary  controlled martingale problem. 

The following proposition establishes one direction of the claim: an admissible state-policy pair $(\bm X, \gamma)$ in $ \mathcal{A}_{\rm cl}(\bm \mu) $ induces an   admissible occupation measure  $\xi$.  Given  $(\bm X, \gamma)$, $\xi$  describes the amount of time the joint state and control process
spends in each region of the state and control space over the time horizon $\sT$. It is a simple application of It\^o's formula.
\begin{Proposition}
\label{prop:cl_measure}
     Suppose (H.\ref{assm:diffusion_state})
holds. Let $\bm \mu \in \cP(\sR^d|\sT)$
and 
let $(\bm X,\gamma)\in \cA_{\rm cl}(\bm \mu)$ be defined on   a probability space $(\Omega, \cF,\sP)$.
 Define 
$\nu\in \cM_+(\sR^d)$
and 
$\xi\in \cM_+(\sT\times \sR^d\times  A)$
such that
for all $F_1\in \cB(\sR^d)$ 
and  $F_2 \in \cB(\sT\times \sR^d\times A)$, 
\begin{equation}
\label{eq:superposition}
\nu(F_1)\coloneqq\sP(X_T\in F_1),\quad \xi(F_2)\coloneqq  \sE^{\sP}\bigg[\int_{\sT\times A } \mathds{1}_{\{(t,X_t,a)\in F_2\}} \gamma(\d a |t,X_t) \d t\bigg].
    \end{equation}
 Then 
$(\nu,\xi)\in \cD_P(\bm \mu)$
and   $J^{\bm \mu}_{\rm cl}(\bm X,  \gamma)=J^{\bm \mu}_P(\nu,\xi) $.

 \end{Proposition}

The converse direction that 
 an   admissible occupation measure  $\xi$
 induces a weak solution of the controlled diffusion process  is  more involved.
 We show that the state marginals  of 
$\xi$ form a measure-valued solution to a controlled FP equation. We then exploit the superposition principle for diffusion processes \cite{trevisan2016well} to lift this measure-valued solution to a solution of the corresponding martingale problem, while preserving its time marginals.
(See Section \ref{sec:proof_primal_equivalence} for  more details). 
  Combined, we have

\begin{Theorem}
\label{thm:primal_diffusion}
Suppose (H.\ref{assm:diffusion_state})  holds, and let $\bm \mu \in \cP(\sR^d|\sT)$. Then the   formulations 
  \eqref{controlled_cl_diffusion} and 
  \eqref{Linear_programming_diffusion}
 have the same value: 
\begin{equation}
\label{eq:representative_player_equivalence_diffusion}
\inf_{(\bm X,\gamma)\in \cA_{\rm cl}(\bm \mu) } J^{\bm \mu}_{\rm cl}(\bm X,\gamma)
 =\inf_{(\nu,\xi)\in \mathcal{D}_P(\bm \mu) }J^{\bm \mu}_P(\nu,\xi).
\end{equation}
\end{Theorem}

\begin{Remark}
\label{remark:primal}    
Theorem \ref{thm:primal_diffusion}  establishes an  equivalent LP for  a continuous-time control problem with  \emph{measurable} coefficients. 
In fact,  
  by Propositions~\ref{prop:cl_measure} and~\ref{prop:measure_diffusion}, every feasible state-policy pair to \eqref{controlled_cl_diffusion}  induces a feasible occupation measure to \eqref{Linear_programming_diffusion} with the same cost, and vice versa. Consequently, a minimizer of either formulation, whenever it exists, induces a minimizer of the other.
All prior  results on such equivalence assume  that the system coefficients are \emph{continuous} 
(see, e.g., \cite{stockbridge1990time, kurtz1998existence})  or  \emph{Lipschitz continuous} in the state variables (see \cite{dumitrescu2021control, djete2022extended}).

Weakening the continuity assumption
as in Theorem \ref{thm:primal_diffusion}
 is crucial for our subsequent analysis for dual formulation of the LP, and in particular  solvability of the dual problem   and characterization  of the MFE set. This is because the optimal policy of the representative player's problem   is only measurable in the state variable, which implies  that the state coefficient under an optimal  policy may be only measurable in the state variable and not necessarily continuous.

\end{Remark}

Note that (H.\ref{assm:diffusion_state})  alone does not guarantee the existence of minimizers for \eqref{controlled_cl_diffusion} or \eqref{Linear_programming_diffusion}. Existence of an optimizer can be ensured  by compactness of the action space together with continuity of the model coefficients \cite{kurtz1998existence,nicole1987compactification}, or by nondegeneracy of the diffusion coefficient yielding sufficient (Sobolev) regularity of the value function \cite{krylov1980controlled}.
In the sequel, we work with (H.\ref{assm:diffusion_state}) 
to establish a  sufficient condition  for MFEs in  Theorem \ref{thm:primal_dual_NE_diffusion_sufficient},
and impose additional regularity condition when deriving the necessary conditions  of MFEs in Proposition~\ref{prop:pde_regularity}.  

\section{Dual formulation and verification of MFEs}

We now derive  the dual formulation of the representative player's primal problem as  a maximization problem over smooth subsolutions of the associated HJB equation.  The   combination of the primal   and   dual solutions will lead to  a novel verification theorem 
for  MFEs.

\paragraph{Derivation of the dual problem.}

To derive the dual formulation of 
\eqref{Linear_programming_diffusion}, we first rewrite the primal problem \eqref{Linear_programming_diffusion}       
using dual pairs of   (infinite-dimensional) topological vector spaces. 

 Let
$\cX =\cM(\sR^d) \times  \cM(\sT\times \sR^d\times A)  $,   
$ \cY= B_b(\sR^d)\times B_b(\sT \times \sR^d\times  A)$,
  $\cW= C^{1,2}_b(\sT\times \sR^d)   $ as defined  in Section \ref{sec:primal_diffusion},
and    $\cZ =\cW^\#$ be the  algebraic dual space of $\cW$ containing all linear functionals from $\cW$ to $\sR$.
And we   equip   the spaces $\cX$ and $\cZ$ with weak topologies induced by the dual pairings
between $\cX\times \cY$ and $\cZ\times \cW$, respectively. More precisely, 
let 
$\langle \cdot,\cdot \rangle_{\cX\times \cY}$ be the bilinear form on $\cX\times \cY$ such that 
for all $(\nu,\xi)\in \cX$ and $(u,\phi)\in \cY$,
$$\big\langle (\nu,\xi),(u,\phi)\big\rangle_{\cX\times \cY} \coloneqq 
\int_{\sR^d }  u(x)\nu (\d x) + \int_{\sT \times \sR^d \times A} \phi(t,x, a)\xi( \d t,\d x, \d a).$$
We also endow $\cX$  with the   weak topology $\sigma(\cX,\cY)$
(called $\sigma$-topology on $\cX$) induced by the dual   pair   $\langle \cdot,\cdot \rangle_{\cX\times \cY}$, i.e., the coarsest topology under which for all   $(u,\phi) \in \cY$,  the linear map $(\nu,\xi)\mapsto \langle (\nu,\xi), (u,\phi) \rangle_{\cX\times \cY} $ is continuous. 
Similarly, we  define the bilinear form $\langle \cdot,\cdot\rangle_{\cZ\times \cW} $ on $\cZ\times \cW$ such that
 $\langle \eta, \psi\rangle_{\cZ\times \cW}\coloneqq  \eta(\psi)$ 
 for all $\psi \in \cW$ and $\eta \in \cZ$.
 We endow $\cZ$ with the   weak topology $\sigma(\cZ,\cW)$
(called $\sigma$-topology on $\cZ$) induced by the dual pairing $\langle \cdot,\cdot\rangle_{\cZ\times \cW} $.   

Using the dual pair  $\langle \cdot, \cdot  \rangle_{\cX\times \cY}$,
the primal problem \eqref{Linear_programming_diffusion}
can be rewritten equivalently as 
\begin{equation}
 \label{eq:primal_diffusion2}
\tag{LP*}
 \inf_{(\nu,\xi)\in  {\cD_P}(\bm \mu) }   {J}^{\bm \mu}_P(\nu,\xi),
 \quad 
 \textnormal{with $
   {J}^{\bm \mu}_P(\nu,\xi)\coloneqq \Big\langle (\nu,\xi), \big(g(\cdot,\mu_T),  f(\cdot,\cdot,\cdot,\mu_\cdot)\big)\Big\rangle_{\cX\times\cY},$}
\end{equation}
with 
 $ \mathcal{D}_{{P}}(\bm \mu)
= \{ 
(\nu,\xi)\in \cX_+
\mid \cL(\nu,\xi) =h 
  \}$,
  where 
  $\cL:\cX  \mapto \cZ$ 
  is a linear map and $h \in \cZ$   such that   for all $\psi\in \cW$, 
  \begin{equation}
  \label{eq:L_diffusion}
 \cL(\nu,\xi)(\psi) 
   \coloneqq \left\langle 
 (\nu,\xi), \big(\psi(T,\cdot), 
  -(\partial_t \psi  + \sL^{\bm \mu}  \psi )\big)\right\rangle_{\cX\times \cY}, \quad   h( \psi) \coloneqq 
   \int_{\sR^d }\psi(0,x) \rho (\d x).
  \end{equation}
The map $\cL$   is   continuous with respect to the $\sigma$-topology as shown in the following lemma.

  \begin{Lemma}
\label{lemma:L_diffusion}
    Assume (H.\ref{assm:diffusion_state}) and let $\bm \mu \in \cP(\sR^d|\sT)$.  
   For each   $(\nu,\xi)\in \cX$ and $\psi\in \cW$, take  $  \cL(\nu,\xi)( \psi)$ from  \eqref{eq:L_diffusion}.   Then  
    $\cL:\cX  \mapto \cZ$  
     is a well-defined, $\sigma(\cX,\cY)$-$\sigma(\cZ,\cW)$-continuous linear map. 

\end{Lemma} 

We next derive the dual problem of \eqref{eq:primal_diffusion2} by applying the abstract duality theory in  Section 3.3 of  \cite{anderson1987linear}. 
By     Lemma \ref{lemma:L_diffusion} and  
\cite[Proposition 4, p.37]{anderson1987linear}, 
the adjoint   $\mathcal L^*: \cW\mapto \cY $
 of $\cL$ is a well-defined linear map 
satisfying  
$$
\langle \mathcal L(\nu,\xi), \psi \rangle_{\cZ\times \cW} =\langle (\nu,\xi), \mathcal L^*(\psi) \rangle_{\cX\times \cY},
\quad \forall \psi\in \cW,  (\nu,\xi)\in \cX. 
$$
The dual formulation for 
  \eqref{eq:primal_diffusion2} is then defined as 
\begin{equation}
\label{eq:dual_diffusion_abstract}
\tag{Dual*}
\sup_{\psi\in \cD_{\cP^*}(\bm \mu)} \langle h, \psi\rangle_{\cZ\times \cW },
    \end{equation}
    with the domain 
\begin{equation}
\label{eq:D_dual_diffusion_abstract}
  \cD_{\cP^*}(\bm \mu) \coloneqq  \{
 \psi\in \cW\mid  
 (g(\cdot,\mu_T),  f(\cdot,\cdot,\cdot,\mu_\cdot))- \cL^*(\psi)
  \in  \cX_+^\# \},
\end{equation}
    where $\cX_+^\#$ is the dual cone of $\cX_+$ in $\cY$ defined  by 
    $$
    \cX_+^\# \coloneqq \{(u,\phi)\in \cY\mid \langle (\nu,\xi), (u,\phi)   \rangle_{\cX\times \cY}\ge 0, \quad \forall (\nu,\xi)\in \cX_+\}.
    $$
Hereafter, we call $\psi\in \cW$ a dual variable and $\psi\in \cD_{P^*}(\bm \mu)$ a feasible/admissible dual variable.

The following proposition presents an explicit formulation of the dual problem  \eqref{eq:dual_diffusion_abstract}.

 \begin{Proposition}
     \label{prop:D_dual_equivalence}
 Suppose (H.\ref{assm:diffusion_state})  holds, and  let $\bm \mu \in   \cP(\sR^d|\sT)$. 
 The dual problem 
 \eqref{eq:dual_diffusion_abstract}
can be written in the following   form:  
\begin{equation}
\label{eq:dual_diffusion}
\tag{Dual}
\sup_{\psi\in \cD_{\cP^*}(\bm \mu)}J_{{P}^*}^{ \bm\mu} (\psi),
\quad 
\textnormal{with $J_{{P}^*}^{ \bm\mu} (\psi) \coloneqq \int_{\sR^d}\psi(0,x)\rho(\d x)$},
\end{equation}
and $\cD_{P^*}(\bm \mu ) \subset \cW$   defined by
\begin{equation} 
\label{eq:D_dual_diffusion}
\cD_{P^*}(\bm \mu ) \coloneqq 
\left\{ 
\psi  \in \cW
\,\middle\vert\, 
\begin{aligned}
& g(x,\mu_T)\ge    \psi(T,x), \, \partial_t\psi(t,x) + \big(\sL^{{\bm \mu}}  \psi \big)(t,x,a)+
f(t,x,a,\mu_t)\ge 0 
\\
&
\textnormal{for all  $(t, x, a)\in \sT\times \sR^d\times A $.}
 \end{aligned}
\right\}.
\end{equation}

 \end{Proposition}

\begin{Remark}
\label{rmk:dual_problem}
The dual problem 
\eqref{eq:dual_diffusion}
can be interpreted as  the supremum over all smooth subsolutions of the following HJB equation:
for all $(t,x)\in \sT\times \sR^d$,
\begin{equation*}
\partial_t V(t,x) +\inf_{a\in A} \left(
 (\sL^{\bm \mu} V)(t, x, a)
 +f(t,x,a,\mu_t) \right)=0,
 \quad   V(T,x)=g(x,\mu_T).
\end{equation*}
Such observations have been made for control problems with Lipschitz coefficients in \cite{lions1983hamilton, fleming1989convex, buckdahn2011stochastic, serrano2015lp}, where the dual problem serves as an intermediate step linking the  value function with viscosity solutions of the HJB equation.  
Here we extend the dual formulation 
to control problems with   measurable coefficients.
As will be shown in Theorems 
\ref{thm:primal_dual_NE_diffusion_sufficient}
and \ref{thm:primal_dual_NE_diffusion_necessary},
a  solution to this dual problem plays an essential role in characterizing  MFEs. 

\end{Remark}

\paragraph{Verification theorem of MFEs.}

To characterize the MFE, first
observe  that  the value of the dual problem is a lower bound for the value of the primal problem,
as shown below.

\begin{Proposition}[Weak duality]\label{prop:weak_duality_diffusion}
    Assume (H.\ref{assm:diffusion_state}) 
and $\bm\mu \in \cP(\sR^d |\sT)$.
  Then
$$
\inf_{(\nu,\xi)\in \cD_{ P}(\bm \mu)}J_{{P}}^{\bm \mu} (\nu,\xi ) \ge \sup_{\psi\in \mathcal D_{P^*}(\bm \mu)} J_{{P}^*}^{\bm \mu} (\psi ).
$$

\end{Proposition}

    Proposition \ref{prop:weak_duality_diffusion} 
    implies the following verification theorem for an optimal control of the representative player's problem.

\begin{Proposition} \label{prop:representative_verification}

    Assume (H.\ref{assm:diffusion_state})  and 
  $\bm \mu \in \cP(\sR^d|\sT)  $.
 If 
 $(\bar \nu,\bar \xi,  \bar \psi)
 \in \cM_+(\sR^d) \times \cM_+(\sT\times \sR^d \times A)\times 
C^{1,2}_b(\sT\times \sR^d)
 $ satisfies the following system of  linear equalities and inequalities: 
    \begin{subequations}
  \label{eq:primal_dual_diffusion}
\begin{align}[left = \empheqlbrace\,]
\begin{split}\label{eq:primal_dual_value_diffusion}
& \int_{\sR^d} g(x,\mu_T )\bar \nu (\d x)+\int_{\sT\times\sR^d\times A} f(t,x,a,\mu_t)\bar \xi(\d t,\d x,\d a) =\int_{\sR^d}\bar  \psi (0,x)\rho(\d x),
\end{split}\\
\begin{split}  \label{eq:primal_constraint_diffusion}
 &\int_{\sR^d}  \psi(T,x) \bar\nu(\d x)  
 - \int_{\sR^d}  \psi(0,x)\rho(\d x)\\
& \qquad =  \int_{\sT \times \sR^d \times A}\Big(\big(\sL^{\bm \mu}  \psi\big)(t,x,a)+(\partial_t  \psi)(t,x)\Big)\bar \xi(\d t, \d x,  \d a),
    \quad \forall \psi \in \cW,
\end{split}\\
\begin{split}\label{eq:dual_constraint_1_diffusion}
    g(x,\mu_T)\ge  \bar  \psi(T,x),  \quad  \forall  x\in\sR^d ,  
   \end{split}
   \\
\begin{split}\label{eq:dual_constraint_2_diffusion} 
 \partial_t \bar \psi(t,x) + \big(\sL^{{\bm \mu}} \bar  \psi \big)(t,x,a)+
f(t,x,a,\mu_t)\ge 0,
     \quad  \forall   (t, x,a) \in \sT\times  \sR^d\times A,  
   \end{split}
\end{align}
\end{subequations}
then 
$(\bar \nu,\bar \xi)\in  \argmin_{(\nu,\xi)\in  {\cD_P}(\bm \mu) }   {J}^{\bm \mu}_P(\nu,\xi) $
and $
  \bar \psi \in  \argmax_{\psi\in  {\cD_{P^*}}(\bm \mu) }   {J}^{\bm \mu}_{P^*}( \psi)$.

\end{Proposition}

Condition \eqref{eq:primal_constraint_diffusion} ensures  that $(\bar \nu,\bar \xi)$
is a feasible solution
to the primal problem, 
Conditions  
\eqref{eq:dual_constraint_1_diffusion}   and 
\eqref{eq:dual_constraint_2_diffusion}   guarantee  
 that $\bar \psi$
 is a feasible solution
 to the dual problem, 
 and Condition  \eqref{eq:primal_dual_value_diffusion} 
 implies $(\bar \nu,\bar \xi)$ and   $\bar \psi$ yield the same value, 
which ensures   the optimality of $(\bar \nu,\bar \xi, \bar \psi)$
from the weak duality property.  

To provide a verification theorem for the MFEs, 
observe that the consistency 
condition in Definition \ref{def:NE}
can be enforced by 
replacing $\bm\mu$ in Proposition \ref{prop:representative_verification}
with the state marginal law of $\bar\xi$. 
This yields the following  verification theorem for an MFE
through   a primal-dual system.

\begin{Theorem}[Primal-dual formulation of MFG]\label{thm:primal_dual_NE_diffusion_sufficient}
    Assume (H.\ref{assm:diffusion_state}).
Let $\bm \mu^*\in \cP(\sR^d|\sT)$ be narrowly continuous,  $ \xi^*\in  \cM_{+}(\sT\times \sR^d \times A)$, and $\psi^*\in C^{1,2}_b(\sT\times \sR^d)$.
  If $(\bm \mu^*,  \xi^*, \psi^*)$ satisfies the following primal-dual system:
\begin{align}
\label{eq:NE_primal_dual_value_diffusion}
& \int_{\sR^d} g(x,\mu^*_T )\mu^*_{T}(\d x)+\int_{\sT\times\sR^d\times A} f(t,x,a,\mu^*_t)\xi^*(\d t,\d x,\d a) =\int_{\sR^d} \psi^*(0,x)\rho(\d x),
\tag{Optimality}
\end{align}
\begin{align}
\label{eq:NE_primal_constraint_diffusion}
 &\int_{\sR^d}  \psi(T,x)\mu^*_T(\d x)  
 - \int_{\sR^d} \psi(0,x)\rho(\d x)\notag\\
& \qquad =  \int_{\sT \times \sR^d \times A}\Big(\big(\sL^{\bm \mu^*}  \psi\big)(t,x,a)+(\partial_t\psi)(t,x)\Big)\xi^*(\d t, \d x,  \d a),
    \quad \forall \psi \in \cW,  
    \tag{Primal Feasibility}
\end{align}
\begin{align}
&    g(x,\mu^*_T)\ge    \psi^*(T,x),  \quad  \forall  x\in\sR^d,
\notag
   \\
& \partial_t\psi^*(t,x) + \big(\sL^{{\bm \mu^*}}  \psi^* \big)(t,x,a)+
f(t,x,a,\mu^*_t)\ge 0,
     \quad  \forall   (t, x,a) \in \sT\times  \sR^d\times A, 
 \label{eq:NE_dual_constraint_2_diffusion} 
 \tag{Dual Feasibility} 
 \end{align}
\begin{align} \label{eq:NE_consistency_constraint_diffusion}
       \int_{\sT\times \sR^d} \psi(t,x) \mu^*_t (\d x) \d t
 =
 \int_{\sT\times \sR^d\times A} \psi(t,x) \xi^*(\d t, \d x, \d a),
 \quad \forall \psi\in \cW,
\tag{Consistency}
\end{align} 
   then
   $\psi^*\in    \argmax_{\psi\in  {\cD_{P^*}}(\bm \mu^*) }   {J}^{\bm \mu^*}_{P^*}( \psi)$, and  there exists 
a   process $\bm X^*$  
 such that 
 $(\bm \mu^*, \bm X^*, \gamma^*)\in \cP(\sR^d|\sT)\times  \cA_{\rm cl}(\bm \mu^*)$ is an MFE,
 with $\gamma^*\in \cP(A|\sT\times \sR^d)$ being the disintegration kernel of $\xi^*$ in Lemma \ref{lemma:disintegration}. 

\end{Theorem}

 \begin{Remark}\label{remark:distinction}

\eqref{eq:NE_primal_constraint_diffusion} ensures  that $   \xi^*$
is an admissible occupation measure
of the primal problem
given 
the      mean field flow $\bm \mu^*$;  
\eqref{eq:NE_dual_constraint_2_diffusion}   ensures
 that $ \psi^*$ 
 is a feasible solution
of the dual problem
given  the    mean field flow $\bm \mu^*$; \eqref{eq:NE_primal_dual_value_diffusion}
  ensures simultaneously that   $   \xi^*$ 
  and   $  \psi^*$ 
  are    optimal solutions
  for the primal and dual formulations, respectively;
finally,  \eqref{eq:NE_consistency_constraint_diffusion}
ensures that the time-space marginal of   $\xi^*$
is consistent with the given mean field flow $\bm \mu^*$.

\begin{Remark}
   A special case of the  primal-dual system 
   in Theorem \ref{thm:primal_dual_NE_diffusion_sufficient} has been studied in \cite{cardaliaguet2015mean, cardaliaguet2015second} for MFGs where  the state process has  uncontrolled diffusion coefficients
   and model coefficients  satisfy additional structural conditions ensuring the uniqueness of MFEs. Specifically, they consider an MFG \eqref{controlled_cl_diffusion}  where  the state space is  a   $d$-dimensional torus $\mathbf T^d$,
   $\rho$ has a density $m_\rho$,
   $b(t,x,a,\mu)=a$, $\sigma(t,x,a,\mu)=\sqrt{2}\bar{\sigma}(x)$, 
$g(x,\mu)=\bar{g}(x)$, 
and 
$f(t,x,a,\mu_t)=\bar{f}(x,m(t,x))+H^*(x,-a)$,
where 
$m(t,\cdot)$ is the density of $\mu_t$,
$H^*$ is strictly convex in $a$, and 
$\bar{f}$ is increasing with respect to the second variable $m$. In this setting, they introduce  the following system
\begin{subequations}
\label{eq:primal_dual_specialcase}
\begin{align}[left = \empheqlbrace\,]
\begin{split}  
\label{eq:optimality_special}
 & \int_{\mathbf T^d} \bar{g}(x)m( T, x)\d x +\int_0^T \int_{\mathbf T^d}\left[ 
 \bar{f}(x,m(t,x))+H^*(x,D_pH(\cdot,D_x \psi(t,x)))\right] m(t,x) \d x\d t 
 \\
 &
 \qquad =\int_{\mathbf T^d} \psi(0,x)m_\rho(x)\d x,
  \end{split}
  \\
  \begin{split}
  \label{eq:FP_special}
  &     \partial_t m= 
      \operatorname{div}[mD_pH(\cdot,D_x \psi) ]
       +\operatorname{tr}\big[D^2_x\big( (\bar{\sigma}\bar{\sigma}^\top)(\cdot)m\big) \big], 
\quad  m(0)=m_\rho,
\end{split}
 \\
\begin{split}  
\label{eq:HJB_special}
 &\partial_t \psi+ 
\operatorname{tr}\big( (\bar{\sigma}\bar{\sigma}^\top)(x)
D^2_x \psi\big)  
-H(x,D_x \psi)
+\bar{f}(x,m(t,x))\ge 0,  
\quad   \bar{g}(x)\ge \psi(T,x),
  \end{split}
\end{align}
 \end{subequations}
 where $H(x,p)\coloneqq 
 \sup_{a\in \sR^d}(a^\top p-H^*(x,a))$.
It is   shown that there exists a unique pair of functions   $(\psi, m)$ which has  appropriate regularity and satisfies \eqref{eq:primal_dual_specialcase} 
 in the sense of distributions. This pair  $(\psi, m)$ is referred to as the weak solution of the MFG, though its relationship to an equilibrium  of the original MFG has not been established. 

Note that \eqref{eq:primal_dual_specialcase}
is a restriction of  the primal-dual system
to strict controls under the structural conditions imposed in \cite{cardaliaguet2015mean, cardaliaguet2015second}.
Indeed, under the assumptions in 
\cite{cardaliaguet2015mean, cardaliaguet2015second}, 
 the representative player has a unique   optimal feedback control
$(t,x)\mapsto -D_p H(x,D_x \psi(t,x))$.
Hence given a solution   $(\psi, m)$
to \eqref{eq:primal_dual_specialcase},
 it holds that  
$  \xi^*(\d t,\d x,\d a)\coloneqq \delta_{-D_p H(x,D_x \psi(t,x))}(\d a)m(t,x)\d x\d t$
and  $ \psi^*\coloneqq \psi$
satisfy the primal-dual system,
where 
\eqref{eq:primal_dual_value_diffusion} corresponds to \eqref{eq:optimality_special},
\eqref{eq:NE_primal_constraint_diffusion}
corresponds to 
\eqref{eq:FP_special},
\eqref{eq:NE_dual_constraint_2_diffusion} corresponds to \eqref{eq:HJB_special},
and 
\eqref{eq:NE_consistency_constraint_diffusion}
follows from the definition of $\xi^*$.

\end{Remark}

\begin{Corollary}\label{cor:complementary_conditions_diffusion}
Given \eqref{eq:NE_primal_constraint_diffusion}  
and \eqref{eq:NE_dual_constraint_2_diffusion}. Then 
\eqref{eq:NE_primal_dual_value_diffusion} 
is equivalent to   the 
following: 
\begin{align}
\label{eq:complementary_diffusion}
\begin{split}
       \int_{\sR^d}(g(x,\mu_T^*)-\psi^*(T,x))\mu_T^*(\d x) =0, 
\\ 
  \int_{\sT \times \sR^d\times A}\left(\partial_t\psi^*(t,x)+\big(\sL^{\bm \mu^*}  \psi^* \big)(t,x,a)+f(t,x,a,\mu_t^*)\right)
  \xi^*(\d t,\d x,\d a)=0.
 \end{split}
\end{align}
Consequently, 
Theorem \ref{thm:primal_dual_NE_diffusion_sufficient} holds when \eqref{eq:NE_primal_dual_value_diffusion} is replaced by \eqref{eq:complementary_diffusion}.
\end{Corollary}

Conditions \eqref{eq:complementary_diffusion}  
and \eqref{eq:NE_dual_constraint_2_diffusion}
suggest the dual variable $\psi^*$ to satisfy the HJB equation and its terminal condition \emph{on the support of  the equilibrium flow $\bm \mu^*$}, namely,
  $\psi^*(T,\cdot)= g( \cdot,\mu_T^*)$ for $\mu_T^*$-a.e.~$x$, 
and
   $\min_{a\in A} \left(\big(\sL^{\bm \mu^*}  \psi^*\big)(t,x,a)+(\partial_t\psi^*)(t,x)+f(t,x,a,\mu^*_t)\right)=0$
   for $\d t \mu_t^*(\d x)$-a.e.~$(t,x)$.
 This is a key distinction between  the primal-dual approach and the existing PDE approach, as will be discussed in detail in 
 Remark \ref{rmk:hjb_fp_pd} and 
 Example \ref{ex:nonsmooth_value}.

\end{Remark}

\section{Strong duality  and characterization of all MFEs}
\label{sec: strongdual}

Theorem \ref{thm:primal_dual_NE_diffusion_sufficient} allows for verifying  a given tuple $(\bm \mu^*,  \xi^*, \psi^*)$ as an MFE  through a primal-dual system. In this section, we will show that this primal-dual system in fact characterizes \emph{all MFEs}, provided that
the equilibrium flow ensures  
strong duality   between the primal problem \eqref{Linear_programming_diffusion} and the dual problem \eqref{eq:dual_diffusion}.
Here,  strong duality means that
given the flow $\bm \mu$, 
both the representative player's primal problem \eqref{Linear_programming_diffusion} and the dual problem \eqref{eq:dual_diffusion}  attain their optimal solutions and yield the same  optimal values. More precisely 

\begin{Definition}
\label{def:strong_duality}
    We say a given flow $\bm \mu \in \cP(\sR^d|\sT)$ ensures strong duality 
    if 
    there exists 
$(\bar{\nu},\bar{\xi})\in \cD_{  P}(\bm \mu)$ 
and $\bar{\psi}\in \cD_{P^*}(\bm \mu)$ 
such that 
$$
J_{{P}^*}^{\bm \mu} (\bar{\psi})=
\sup_{\psi\in \mathcal D_{P^*}(\bm \mu)} J_{{P}^*}^{\bm \mu} (\psi )
=\inf_{(\nu,\xi)\in \cD_{ P}(\bm \mu)}J_{{P}}^{\bm \mu} (\nu,\xi ) 
= J_{{P}}^{\bm \mu} (\bar{\nu},\bar{\xi}).
$$
\end{Definition}

The following theorem 
shows that, with   the primal-dual system in Theorem \ref{thm:primal_dual_NE_diffusion_sufficient}, strong duality characterizes all MFEs.

\begin{Theorem}\label{thm:primal_dual_NE_diffusion_necessary}
    Suppose (H.\ref{assm:diffusion_state})
holds.
Assume  $(\bm \mu^*, \bm X^*, \gamma^*)\in \cP(\sR^d |\sT) \times  \cA_{\rm cl}(\bm \mu^*) $ 
    is an MFE  
    and $\bm \mu^*$ ensures strong duality.
 Define
 $\xi^*\in \cM_+(\sT\times \sR^d\times  A)$
such that
$\xi^*(F )\coloneqq   \int_{F  } \gamma^*(\d a|t,x)\mu^*_t (\d x) \d t$
for all  $F  \in \cB(\sT\times \sR^d\times A)$.
Then 
$\argmax_{\psi\in  {\cD_{P^*}}(\bm \mu^*) }   {J}^{\bm \mu^*}_{P^*}( \psi)$ is nonempty,
and for any   
$\psi^*\in    \argmax_{\psi\in  {\cD_{P^*}}(\bm \mu^*) }   {J}^{\bm \mu^*}_{P^*}( \psi)$,
the triple $(\bm \mu^*,   \xi^*, \psi^*)$ satisfies the primal-dual system in Theorem \ref{thm:primal_dual_NE_diffusion_sufficient}. 

\end{Theorem}

\paragraph{Solvability of dual problem.}

To guarantee the strong duality, it is critical to identify conditions which ensure that the dual problem attains its maximizer. This is known as the solvability of the dual problem.

Recall that
for any given flow $\bm \mu$,
the dual problem involves taking the supremum over all smooth subsolutions of the following HJB equation  (see Remark \ref{rmk:dual_problem}):
\begin{equation}
\label{eq:hjb_diffusion}
\partial_t V(t,x) +\inf_{a\in A} \left(
 (\sL^{\bm \mu} V)(t, x, a)
 +f(t,x,a,\mu_t) \right)=0,
 \quad   V(T,x)=g(x,\mu_T).
\end{equation}
Since the pointwise maximum of subsolutions is known to be a viscosity solution \cite[Theorem 4.1]{crandall1992user},
a natural sufficient condition for the solvability of the dual problem is ensuring that the HJB equation \eqref{eq:hjb_diffusion} admits a smooth viscosity solution, or equivalently a classical solution.

To this end, observe  that
the regularity of the solution to  
\eqref{eq:hjb_diffusion} depends crucially on  
the regularity of 
$(b,\sigma, f)$ with respect to the measure component, 
and 
 the   regularity of  the given mean field flow  $t\mapsto   \mu_t $. In the sequel,
we   focus on   flows
  $\bm\mu =(\mu_t)_{t\in \sT} $
  of the  form: 
  \begin{equation}
  \label{eq:admissible_flow}
 \mu_t =\cL^\sP(X_t),
\quad \forall t\in \sT,
\quad \textnormal{
with $\cL^\sP(X_0)=\rho$
and 
$  X_t =X_0+\int_0^t \bar{b}_s \d s +\int_0^t \bar{\sigma}_s  \d W_s $,}     
  \end{equation}
where 
  $\bar b:\sT\times \Omega\to \sR^d$ and $\bar \sigma:\sT\times \Omega\to \sR^{d\times d}$
are  bounded  $\sF$-progressively measurable   processes
 defined on  
  a filtered probability space $(\Omega, \cF,\sF,\sP)$,
and 
$W$
is   a   $d$-dimensional $\sF$-Brownian motion on $\Omega$. 
Define the space 
$\cU_\rho $ of admissible flows by
\begin{equation}
\label{eq:U_rho}
\cU_\rho \coloneqq
\{\bm \mu \in \cP(\sR^d|\sT)
\mid \textnormal{$\bm \mu=(\mu_t)_{t\in \sT}$
satisfies \eqref{eq:admissible_flow}}
\}.
\end{equation}
 The class $\cU_\rho$
 contains all mean field state distributions induced by a given   policy
 (cf.~\eqref{eq:state_diffusion}),
 and is sufficient to characterize all MFEs. 

 We now show a  flow in 
 $\cU_\rho$ ensures
  strong duality   under  the following   
conditions.

\begin{Assumption}
\label{assum:duality}
   For any  $\bm \mu\in \cU_\rho$,
    \begin{enumerate}[(1)]

\item 
\label{item:pde_regularity}
 There exists $V\in  C^{1,2}_b(\sT\times \sR^d)$ satisfying the    HJB equation 
 \eqref{eq:hjb_diffusion},  and  a measurable function $\phi :\sT\times \sR^d\to A$ such that 
$\phi (t,x) \in \arg\min_{a\in A} \left(
 (\sL^{\bm \mu} V)(t, x, a)
 +f(t,x,a,\mu_t) \right)$
 for all $(t,x)\in \sT\times \sR^d$.

    \item \label{item:feasible}
      Define $(t,x)\mapsto \gamma (\d a|t,x) \coloneqq \delta_{\phi(t,x)}(\d a)$ with $\phi$ given in 
(H.\ref{assum:duality}\ref{item:pde_regularity}). Then the SDE \eqref{eq:state_diffusion} with the policy $\gamma$ admits a weak solution. 

    \end{enumerate}

\end{Assumption}

\begin{Theorem}\label{thm:strong_duality_diffusion}
Assume (H.\ref{assm:diffusion_state}) and  (H.\ref{assum:duality}).
Then any 
 $\bm \mu \in  \cU_\rho$
 ensures strong duality   between  \eqref{Linear_programming_diffusion} and  \eqref{eq:dual_diffusion}
 (cf.~Definition \ref{def:strong_duality}).
In particular, 
 Theorem \ref{thm:primal_dual_NE_diffusion_necessary}
provides  a necessary condition for all MFEs.  

\end{Theorem}

\begin{Remark}
  \label{rmk:strong_duality}  

Theorem \ref{thm:strong_duality_diffusion} establishes two crucial aspects of the representative player's control problem. First, it proves that the optimal value of the primal problem \eqref{Linear_programming_diffusion} is equal to that of the dual problem \eqref{eq:dual_diffusion}. Second, it ensures that both problems admit optimal solutions.

While the first aspect has been previously studied for continuous-time control problems
(see e.g., \cite{fleming1989convex, buckdahn2011stochastic, serrano2015lp}), 
 the solvability of the associated dual problem,
 to the best of our knowledge, has not   been established before. 
 As demonstrated in Theorem \ref{thm:primal_dual_NE_diffusion_necessary}, a solution to the dual problem plays a crucial role in characterizing the set of MFEs, as it  serves as an adjoint variable characterizing the representative player's optimal state-control distribution. 

Although (H.\ref{assum:duality})   holds for a broad class of MFGs with nondegenerate diffusion coefficients (see Proposition~\ref{prop:pde_regularity}), it is nevertheless a sufficient condition for strong duality. Indeed, while (H.\ref{assm:diffusion_state})  with merely measurable coefficients may be  too weak to ensure strong duality in general, Example~\ref{ex:nonsmooth_value} shows that strong duality may still hold even when the associated HJB equation admits only a viscosity solution and not classical solution, therefore (H.2) fails. Identifying weaker conditions for a complete characterization of all MFEs is left for future work.

\end{Remark}

We next present explicit conditions on the model coefficients that ensure  (H.\ref{assum:duality}).
To this end, we first establish 
a new version of  the
Evans--Krylov theorem that yields 
a global $C
^{2,\alpha}$ regularity
  for uniformly parabolic
HJB equations on $\sT\times\sR^d$ with H\"older continuous
coefficients. Compared with the classical result in
\cite[Theorem~6.4.3, p.~301]{krylov1987nonlinear},
our theorem relaxes the Lipschitz continuity assumptions therein 
and extends the spatial setting from bounded domains
to the whole space. 
Besides its independent interest as a PDE regularity result, it
provides an essential  tool for analyzing MFGs with general
dependence on the mean field measure; see Remark \ref{rmk:holder_measure}.

   \begin{Theorem}
[Global $C^{2,\alpha}$ regularity for HJB equations]
\label{thm:evans-krylov}

Let 
$A$ be a non-empty set,
and $S_T=\mathbb T\times \sR^d$.
Let  $\kappa, K, L, L_f>0$ and $ \eta\in (0,1)$
be  
such that $g\in C_b^{2+\eta}(\sR^d)$, and 
for all $a\in A$, 
$(\Sigma^a, b^a, f^a): S_T \to\sS^d\times \sR^d\times \sR$
satisfies 
$\kappa I_d\preceq \Sigma^a(t,x)\preceq KI_d$  for all $(t,x)\in S_T$, 
$\sup_{a\in  A}\|\Sigma^a\|_{C_b^{\eta/2,\eta}(S_T)} \le L$, $  
\sup_{a\in A} \|b^a\|_{C_b^{\eta/2,\eta}(S_T)}
 \le L$,
and 
$\sup_{a\in  A}
\|f^a\|_{C_b^{\eta/2,\eta}(S_T)}
\le L_f$. Define 
the  Bellman operator $F:S_T\times \sR^d\times \sS^d\to \sR$ by 
\begin{equation}
\label{eq:hamiltonian_regularity}
 F(t,x,p,M)\coloneqq \inf_{a\in A}
      \left \{\tfrac12\tr(\Sigma^a(t,x)M)+b^a(t,x)^\top p+f^a(t,x)\right\}.
\end{equation}
Then there exists $\alpha_0\in (0,1)$ such that 
 for all $\alpha\in (0,\alpha_0)$, there exists a unique solution $V\in C_b^{1+\alpha/2, 2+\alpha}(S_T)$ satisfying  for all 
 $(t,x)\in (0,T)\times \sR^d$,  
 \begin{equation}
 \label{eq:hjb_holder}
\partial_t V(t,x)+F(t,x,D_xV(t,x),D^2_xV(t,x))=0,\quad V(T,x)=g(x).
\end{equation}

  \end{Theorem}

Based on Theorem \ref{thm:evans-krylov},
we verify  
  (H.\ref{assum:duality}) for MFGs with non-degenerate diffusion coefficients.  
Recall that  by \cite[Theorem 11.3.3]{dudley2002real},
the weak topology of $\cP(\sR^d)$ can be metrized by 
the bounded Lipschitz metric 
$ 
\d_{\rm BL}(P,Q)\coloneqq
\sup\{  
       \left|\int_{\sR^d}f (\d P-\d Q) \right|\mid \|f\|_{\rm BL}\le 1\},
$ 
where 
$\|f\|_{\rm BL}\coloneqq \|f\|_\infty+\sup_{x,y\in \sR^d,x\not = y}\frac{|f(x)-f(y)|}{|x-y|}$.

\begin{Proposition}
\label{prop:pde_regularity}
Suppose (H.\ref{assm:diffusion_state})  
holds, 
$A$ is compact, and there exists    $\lambda, L>0$ and  $\beta\in (0,1]$ such that  
for all
$t,t'\in \sT$, $x,x'\in \sR^d$, $a\in A$ and $  \mu,\mu'\in \cP(\sR^d)$,
 $  (\sigma\sigma^\top) (t,x,a,\mu)   \succeq \lambda I_d $, $a\mapsto 
(b(t,x,a,\mu), \sigma(t,x,a,\mu))$
is   continuous, 
$a\mapsto f(t,x,a,\mu)$
is lower-semicontinuous, 
\begin{align*}
&|h(t,x,a,\mu)-h(t',x',a,\mu')|
\le
L\left(|t-t'|^{\beta/2}
+|x-x'|^{\beta}
+\d_{\rm BL}(\mu,\mu')^{\beta}
\right),
\quad \forall h\in \{b,\sigma, f\},
\end{align*}
and 
$  g(\cdot,\mu) \in C^{2+\beta}_b(\sR^d)$.
Then  Condition 
(H.\ref{assum:duality}) is satisfied, 
and hence 
 Theorem \ref{thm:primal_dual_NE_diffusion_necessary}
provides  a sufficient and  necessary condition for all MFEs.  

\end{Proposition}

\begin{Remark}[Dependence on the mean field measure]
\label{rmk:holder_measure}

Proposition \ref{prop:pde_regularity}  characterizes the set of   MFEs for a broad class of MFGs whose coefficients depend H\"older continuously on the   measure. To accommodate this setting, it is essential to analyze \eqref{eq:hjb_diffusion} using  Theorem~\ref{thm:evans-krylov}, which allows H\"older continuous coefficients in time. 

Note that even when the coefficients are Lipschitz continuous
in the measure variable, 
the coefficients in \eqref{eq:hjb_diffusion} need not be Lipschitz continuous in time, since the flow
$\sT\ni t\mapsto\mu_t\in\cP (\sR^d)$
is  only  $1/2$-H\"older continuous. 
Applying the classical Evans–Krylov regularity theorem \cite[Theorem~6.4.3, p.~301]{krylov1987nonlinear}, which requires Lipschitz continuity in time, would therefore necessitate additional structural assumptions, such as 
$$
h(t,x,a,\mu)=\bar{h}\left(t,x,a,\int_{\sR^d}\hat{h}(t,y)\mu(\d y)\right),
\quad \forall h\in \{b,\sigma, f\},  
$$
for sufficiently regular $\bar h$ and $\hat h$ (see \cite{dumitrescu2021control}). 
Theorem~\ref{thm:evans-krylov} is developed   to remove such structural restrictions and accommodate the natural temporal regularity of the mean field flow 
  of $t\mapsto \mu_t$. 
\end{Remark}

\section{Comparison with existing approaches for MFGs}\label{sec:distinction}

{ 
Having established   the equivalence between the set of MFEs   and solutions to the  primal-dual system,  we  compare in this section the primal-dual characterization with existing approaches for MFGs,  in particular the coupled HJB-FP (see e.g., \cite{huang2006large, lasry2007mean})
and the BSDE approaches (for instance, \cite{carmona2013probabilistic,cardaliaguet2019master}).

\subsection{Consistency with the HJB–FP approach under a unique  optimal control}

 We first show that when the associated Hamiltonian has a unique optimizer and the equilibrium flow has full support, 
 the primal-dual system in Theorem \ref{thm:primal_dual_NE_diffusion_sufficient}  
is consistent with the 
HJB-FP approach.
We then present two examples without these restrictions, where this primal-dual approach remains applicable and the HJB-FP and the BSDE approaches fail, as the latter two rely on the uniqueness of the optimizer for the Hamiltonian. Moreover,
when the equilibrium    flow 
does not have full support, 
an admissible dual variable  
  is not required to  satisfy   the HJB equation 
  for all states, hence  ensures the applicability of the primal-dual approach in characterizing MFEs even when the HJB equation does not admit a classical or even a continuous solution. 
}

 Specifically, recall that  
the  HJB-FP approach 
assumes 
 the existence of a \textit{unique}
 function 
$\pi:   \sT\times \sR^d\times \sR^d\times
\sR^{d\times d}
\times \cP(\sR^d) \mapto A$ satisfying  
\begin{equation}
\label{eq:feedback_map}
 \pi (t,x, y,z,\mu) =\argmin_{a\in A}H(t,x,y,z,a, \mu),
\end{equation} 
where $H: \sT\times \sR^d\times \sR^d\times
\sR^{d\times d}
\times A\times \cP(\sR^d) \to \sR$ is the Hamiltonian defined by  
$$
H(t,x,y,z,a,\mu)\coloneqq 
\frac{1}{2}\operatorname{tr}\Big(\big(\sigma \sigma^\top\big) (t,x,a,\mu) z \Big)
        +b(t,x,a,\mu)^\top y         
        +f(t,x,a,\mu);
$$
 see 
 \cite[(H.5)]{huang2006large},
 \cite[Equation (13)]{cardaliaguet2019master},
 and \cite[Assumption 3.1.(b)]{bayraktar2019rate}. 
Given  $\pi$, consider the following coupled HJB-FP system: for all $(t,x)\in \sT\times \sR^d$,  
\begin{subequations}
\label{eq:FP+HJB}
\begin{align}[left = \empheqlbrace\,]
\begin{split}  \label{eq:HJB}
 \partial_t V(t,x)&+\min_{a\in A}
 H\big(t,x, D_x V (t,x),
 D^2_x V (t,x),a,\mu_t\big)=0,  
\quad V(T,x)= g(x, \mu_T),
  \end{split}
  \\
  \begin{split}
  \label{eq:FP}
       \partial_t \mu_t&= 
     -\operatorname{div}[b(t,\cdot,  {\phi}(t,\cdot),\mu_t)\mu_t]
       +\frac{1}{2}\operatorname{tr}\big[D^2_x\big( (\sigma\sigma^\top)(t,\cdot,{\phi}(t,\cdot),\mu_t)\mu_t\big) \big], 
\quad  \mu_0=\rho, 
\end{split}\\
  \begin{split}
      {\phi}(t,x) & = \pi\big(t,x,
 D_x V (t,x), D^2_x V (t,x),\mu_t\big).
  \end{split}
\end{align}
 \end{subequations}  
 When $\pi$ is sufficiently regular, \eqref{eq:FP+HJB}
 admits a classical solution, which corresponds  to  an MFE. 
Note that in this HJB-FP approach, the FP equation \eqref{eq:FP}  depends on the unique choice of $\pi$ satisfying \eqref{eq:feedback_map}.

{ 

The following theorem shows that when the Hamiltonian has a unique optimizer and the equilibrium flow has full support, 
the primal-dual system in Theorem   \ref{thm:primal_dual_NE_diffusion_sufficient}  
coincides with a weak formulation of the 
HJB-FP system  \eqref{eq:FP+HJB}.

 \begin{Theorem}
 \label{thm:HJB-FP-PD}
     Suppose (H.\ref{assm:diffusion_state})
holds. 
Let $\bm \mu^*\in \cP(\sR^d|\sT)$ be narrowly continuous and $\psi^*\in C^{1,2}_b(\sT\times \sR^d)$.
Assume   that
  there exists  a  {unique}
  measurable function 
$\phi:   \sT\times \sR^d   \mapto A$ such that   
\begin{equation}
\label{eq:unique_minimizer}
 \phi (t,x) =\argmin_{a\in A}H\big(t,x,D_x \psi^* (t,x), D^2_x \psi^* (t,x),a, \mu^*_t  \big),
 \quad \forall (t,x)\in \sT\times \sR^d,
\end{equation} 
the functions  
$ x\mapsto g(x,\mu^*_T)$
and 
$ (t,x) \mapsto \min_{a\in A} H \big(t,x,D_x \psi^* (t,x), D^2_x \psi^* (t,x),a, \mu^*_t  \big) $  are    continuous,
and
\begin{equation}
\label{eq:full_support}
\operatorname{supp}(\mu^*_t)
 \coloneqq {\{x\in  \sR^d \mid \mu^*_t(\sB_r(x))>0, \; \forall r>0\}}
=\sR^d,\quad   
 \forall t\in (0,T],    
\end{equation}
where $\sB_r(x) \coloneqq  \{y\in \sR^d\mid |y-x|< r\}$.
Then 
for any   $ \xi^*\in  \cM_{+}(\sT\times \sR^d \times A)$,
the triple $(\bm 
\mu^*, \xi^*,\psi^*)$
satisfies
  the  primal-dual system in Theorem \ref{thm:primal_dual_NE_diffusion_sufficient} 
  if and only if
$\xi^*(\d t,\d x,\d a)=\delta_{\phi(t,x)}(\d a)\mu_t^*(\d x)\d t$
and $(\bm\mu^*, \psi^*)$
satisfies the following system:
  \begin{subequations}
  \label{eq:comparision_FP_HJB}
\begin{align}[left = \empheqlbrace\,]
\begin{split} \label{eq:comparison_FP}  &\int_{\sR^d}  \psi(T,x)\mu^*_T(\d x)  
 - \int_{\sR^d} \psi(0,x)\rho(\d x)\\
& \qquad =  \int_{\sT \times \sR^d }\Big(\big(\sL^{\bm \mu^*}  \psi\big)(t,x,\phi(t,x))+(\partial_t\psi)(t,x)\Big)\mu^*_t( \d x)\d t,
    \quad \forall \psi \in \cW,
\end{split}\\
\begin{split} \label{eq:comparison_HJB}
&\partial_t \psi^*(t,x)+\min_{a\in A} H\big(t,x,D_x \psi^* (t,x), D^2_x \psi^* (t,x),a, \mu^*_t  \big)=0,  
\quad \forall (t,x)\in  \sT\times \sR^d.
\\
&\psi^*(T,x)= g(x, \mu^*_T), \quad \forall x\in \sR^d.
   \end{split}
\end{align}
\end{subequations}

\end{Theorem}

\begin{Remark}
\label{rmk:hjb_fp_pd}
    Note that the equivalence between the primal-dual system and the HJB-FP system, as established in Theorem \ref{thm:HJB-FP-PD}, relies crucially on Condition \eqref{eq:unique_minimizer} regarding the uniqueness of the Hamiltonian's optimizer, as well as  the full support condition \eqref{eq:full_support} on the equilibrium flow $\bm \mu^* $. 
 Indeed, 
   Condition \eqref{eq:unique_minimizer}
   ensures that 
   for any $(t,x)\in \sT\times \sR^d$,
the Markov policy derived from  $\xi^*$  is concentrated on  
the unique optimizer $\phi(t,x)$ of the Hamiltonian, 
which yields    the  optimal strict policy. 
The full support condition 
\eqref{eq:full_support}
and Corollary \ref{cor:complementary_conditions_diffusion}
   ensure   that 
   the   dual variable $\psi^*$ 
satisfies the HJB equation \emph{everywhere}.
Equation \eqref{eq:comparison_FP} 
then 
indicates that $ \bm\mu^*=(\mu^*_t)_{t\in \sT} $ is a weak solution to the FP equation \eqref{eq:FP} for the controlled state process, whose coefficients involve the optimal   policy  $ \phi $ and the equilibrium   flow $  (\mu^*_t)_{t\in \sT}$. 
\end{Remark}

}

  To ensure the Hamiltonian has a  unique optimizer, existing works assume that $A\ni a\mapsto H(t,x,y,z,a,\mu)$ is strictly convex, which necessitates the convexity of the action set $A$, the linearity of $b$ and $\sigma\sigma^\top$   in $a$, and the strict convexity of the running cost function $f$.

\subsection{Examples beyond the scope of  classical approaches}

  Unfortunately,
  when the Hamiltonian   is nonconvex, such is the case when $A$ is finite, \eqref{eq:feedback_map} 
  may have multiple optimizers. 
  In this case, the coupled system \eqref{eq:FP+HJB} requires selecting a specific feedback map $\pi$ and therefore provides only {\it sufficient}  conditions for an MFE associated with that particular choice of $\pi$.  Moreover, by restricting attention to strict policies $\phi:\mathbb T\times\mathbb R^d\to A$, \eqref{eq:FP+HJB} may fail to capture relaxed MFEs, even when such equilibria exist.
Similarly,  the BSDE approach also assumes  the convexity of the Hamiltonian 
 to characterize the optimal control of the representative player 
(\cite[Lemma
2.1]{carmona2013probabilistic},
 \cite[Assumption (A4), p.~156]{carmona2018probabilistic}).

In contrast, our primal-dual approach  
 characterizes the set of \emph{all} MFEs,  including the relaxed MFEs, 
 without assuming the uniqueness of the feedback map $\pi$: it  
avoids the prior selection of $\pi$ by 
 including the   policy $\phi$ as part of the solution $\xi$.

 The following example illustrates the advantage of the primal-dual approach over the HJB-FP formulation
  when \eqref{eq:feedback_map}  has multiple optimizers, and when the MFE is induced by a relaxed policy.

\begin{Example}[MFG with discrete actions and relaxed equilibria]
\label{ex:discrete_action}

Consider a  stochastic   coordination game
where  players are rewarded for aligning their states with the aggregate population tendency.  
Let $d=1$, $A=\{-1,1\}$, and $\rho=\delta_0$. Set $b(t,x,a,\mu)=a$ and $\sigma(t,x,a,\mu)=\sigma_0$ for some $\sigma_0>0$. Let $h\in C_b(\sR)$ be odd and strictly increasing, and set the cost functions as $f(t,x,a,\mu)=-\kappa M(\mu)h(x)$ and $g(x,\mu)=-\lambda M(\mu)h(x)$, where $\kappa,\lambda>0$,
and  $M(\mu)\coloneqq\int_{\sR}h(y)\mu(\d y)$.

The game admits at least three distinct MFEs. The constant policies $\phi^+\equiv1$ and $\phi^-\equiv-1$ induce two strict MFEs, with respective mean field flows $\mu_t^+=\mathcal N(t,\sigma_0^2t)$ and $\mu_t^-=\mathcal N(-t,\sigma_0^2t)$ for all $t\in \sT$.\footnote{ $\mathcal N(\mu, \sigma^2)$ denotes the normal distribution with mean $\mu$ and variance $\sigma^2$.} Moreover, the relaxed policy $\bar \gamma(\d a|t,x)=\frac12\delta_{-1}(\d a)+\frac12\delta_1(\d a)$ induces another MFE, with the mean field flow $\bar\mu_t=\mathcal N(0,\sigma_0^2t)$ for all $t\in \sT$.

Along the relaxed equilibrium flow   $\boldsymbol {\bar \mu}$, the HJB equation \eqref{eq:HJB} admits the classical   solution $V(t,x)=0$, which satisfies
\[
H(t,x,D_xV(t,x),D_x^2V(t,x),a,\bar {\mu}_t )=0,\quad \forall(t,x,a)\in[0,T]\times\sR\times A.
\]
Hence, the optimal feedback map in \eqref{eq:feedback_map} is non-unique. Moreover, one can show that 
this equilibrium flow ${\boldsymbol{\bar \mu}}$ cannot be induced by any strict control $\phi:\sT\times \sR\to A$. Thus, the classical HJB--FP system \eqref{eq:FP+HJB} cannot recover this relaxed MFE.

In contrast, the primal-dual system in Theorem~\ref{thm:primal_dual_NE_diffusion_sufficient} recovers this MFE through the smooth dual function $\psi(t,x)=0$ and the occupation measure $\xi(\d t,\d x,\d a)=\d t\,\bar \mu_t(\d x)\left(\frac12\delta_{-1}(\d a)+\frac12\delta_1(\d a)\right)$.

\end{Example}

{
Moreover,
when the equilibrium    flow 
$\bm \mu^*$  
does not have full support, 
the  dual variable $\psi^*$
 in the primal-dual system 
  is not required to  satisfy   the HJB equation \eqref{eq:HJB}  
  for all states.} 
  This distinctive feature of the primal-dual approach ensures its applicability in constructing  MFEs even when the HJB equation does not admit a classical solution. 
  This distinction is reflected by \eqref{eq:complementary_diffusion} and further illustrated in the following example.

\begin{Example}
[MFG with non-differentiable value functions]
\label{ex:nonsmooth_value}
Consider a  
stochastic tracking game where 
each player controls
both the drift and the volatility of their state to match the terminal population
aggregate.  
Let $d=1$,  $A=[-1,1]^2$, and  $\rho=\delta_{\bar x_0}$ for some $\bar x_0\in \sR$.
Let $h\in C_b(\sR)$ satisfy $h(x)=x$ on
$[\bar x_0-T,\bar x_0+T]$.
Set 
$ b(t,x,a,\mu)=a_1$, 
$\sigma(t,x,a,\mu)=a_2$,
$f(t,x,a,\mu) =a_2^2$,
and $g(x,\mu) =q\bigl(x-M(\mu)\bigr) $,  
where $M(\mu)\coloneqq \int_{\sR}h(y)\,\mu(\mathrm dy)$ and
$q(z) \coloneqq 1-e^{-|z|}$.  

This MFG admits infinitely many MFEs, since for each $a\in[-1,1]$,
the constant policy
$\phi^a \equiv (a,0)$, 
the resulting   state process  
$X_t^a\coloneqq \bar x_0+a t $,
and the flow 
  $\mu^a_t=\delta_{X^a_t} $ for $t\in \sT$ constitute an MFE. 
Along the   flow  $\boldsymbol{\mu}^a$, 
write  $m^a\coloneqq \bar x_0+aT$, and   
the   HJB equation \eqref{eq:HJB}
simplifies to
\begin{equation}
 \partial_t V -|D_xV |-\bigl( 1+\tfrac12 D^2_xV \bigr)_-=0,\quad 
 V(T,x)=q(x-m^a),
\label{eq:tracking_hjb}
\end{equation}
which   admits a unique viscosity solution 
$ 
 V(t,x)\coloneqq q((|x-m^a|-(T-t))_+ )
 $, 
 $ (t,x)\in \sT\times\sR$.\footnote{
 We denote by \(y_+\) and \(y_-\) the positive and negative parts of \(y\in\mathbb R\),  respectively.
 Moreover, we refer to \cite{diehl2014regularity} and the references therein for the definition and uniqueness of viscosity solutions.} 

 Since
 for all $t\in \mathbb T$, 
 $V(t,\cdot)$ is not differentiable at 
$x =m^a\pm(T-t)$,  
\eqref{eq:tracking_hjb}
does not admit a classical solution  in
  $  C^{1,2}_b(\sT\times\sR)$.
Hence, the HJB–FP system \eqref{eq:FP+HJB} based on    classical solutions cannot be used to characterize these equilibria.

In contrast, the primal-dual system in Theorem~\ref{thm:primal_dual_NE_diffusion_sufficient} can capture  these MFEs 
through 
$ 
\psi^*(t,x)=0$ and $
\xi^*(\d t,\d x,\d a)
=\d t\,\mu_t^a(\d x)\delta_{(a,0)}(\d a) 
$. 
Note that $\psi^*\le V$
and 
 $\psi^*(T,\cdot)=V(T,\cdot)$
only at $x=m^a$, i.e., on the support of the equilibrium measure $\mu^a_T$. 

\end{Example} 

 \section{Proofs of main results}
\label{sec:proofs}

\subsection{Proof of Theorem \ref{thm:primal_diffusion}}
\label{sec:proof_primal_equivalence}

Given Proposition \ref{prop:cl_measure}, the  key is to  prove that an admissible occupation measure induces a weak solution of the controlled diffusion process. 

To this end, we first connect an occupation measure $\xi$ with a Markov policy $\gamma$  using  the following disintegration theorem  
 (see \cite[Theorem 14.D.10]{baccelli2020random}).

\begin{Lemma}
\label{lemma:disintegration}
Let $(X, \cX)$ and $(Y, \cY)$ be two measurable spaces, and assume $(Y,\cY)$ is Polish.
Let $\mu\in \cM_+(X\times Y)$ 
and  $\mu^{X}$ be the marginal measure of $\mu$ on $(X,\cX)$ defined by 
 $\mu^X(F) =\mu(F\times Y)$ for all $F\in \cX$. 
 Then there exists a $\mu^X$-a.e.~unique $\kappa \in \cP(Y|X)$ such that
for all  $ F\in \cX$ and  $E\in \cY$,
$ 
\mu(F \times E)=\int_F \kappa(E|x) \mu^X(\d x)$.
The kernel $\kappa$ is called the 
 disintegration (probability) kernel
 of $\mu$ with respect to $\mu^X$. 

\end{Lemma}

Next, using Lemma \ref{lemma:disintegration},
any $\xi\in \cM_+(\sT\times \sR^d\times A)$ 
 can be represented as 
\begin{equation}
\label{eq:decomposition}
 \xi(\d t,\d x,\d a)=
 \gamma(\d a|t,x)\xi^{\sT\times \sR^d}(\d t,\d x)
 = \gamma(\d a|t,x)m^X(\d x|t) 
 \xi^\sT(\d t),
 \end{equation}
 with $ \gamma \in \cP(A|\sT\times \sR^d)$ and 
 $m^X\in \cP (\sR^d|\sT)$. 
We will write $m^X_t (\d x) =m^X(\d x|t)$ interchangeably.  
 Given the decomposition \eqref{eq:decomposition},
we then identify (a version of) the curve  $m^X=(m^X_t)_{t\in \sT}\subset \cP(\sR^d)$
as a narrowly continuous weak solution of the   
following  equation:
 \begin{align}
 \label{eq:FP_mu_gamma}
     \partial_t m^X_t  = (\sL^{\bm \mu,\gamma})^\dagger m^X_t,
     \quad \forall t\in \sT;
     \quad m^X_0=\rho, 
 \end{align}
where the  operator $\sL^{\bm \mu,\gamma} $  is defined by 
 \begin{align}
 \label{eq:L_mu_gamma}
        \big(\sL^{\bm \mu,\gamma}  \psi\big)(t,x)\coloneqq \frac{1}{2}\operatorname{tr}\Big(\big(\sigma^{\bm \mu, \gamma}(\sigma^{\bm \mu, \gamma})^\top\big) (t,x) D^2_x\psi (t,x)\big)
        +b^{\bm \mu, \gamma}(t,x)^\top 
        D_x\psi (t,x)
 \end{align}
    with the policy $\gamma$ in \eqref{eq:decomposition} 
    and 
    functions  $b^{\bm \mu, \gamma}$ and $\sigma^{\bm \mu, \gamma}$   given  in \eqref{eq:def_bsigmamu},
    and 
    $(\sL^{\bm \mu,\gamma})^\dagger $ is the adjoint operator of $\sL^{\bm \mu,\gamma} $. 

\begin{Proposition}\label{prop:cont_margin}
    Suppose (H.\ref{assm:diffusion_state})
holds. Let $\bm \mu \in \cP(\sR^d|\sT)$ and $(\nu,\xi)\in  \mathcal{D}_{{P}}(\bm \mu)$.  
    \begin{enumerate}[(1)]
        \item\label{item:timemargin_Les_cont} 
The    time marginal measure       $\xi^\sT $ of $\xi$  is  
 the  restriction  of the Lebesgue measure on the $\sigma$-algebra $\cB(\sT)$.      

        \item \label{item:statemargin}
    Let $\gamma\in\cP(A|\sT\times \sR^d)$ be the disintegration of $\xi$ with respect to $\xi^{\sT\times \sR^d}$, and $m^X \in \cP (\sR^d|\sT)$    
    be the disintegration of $\xi^{\sT\times \sR^d}$ with respect to $\xi^\sT$. Then 
$m^X =(m^X_t)_{t\in \sT}\subset \cP(\sR^d)$ is a weak solution of \eqref{eq:FP_mu_gamma}, 
    in the sense that 
   for all $\psi \in \cW$,
        \begin{align}\label{eq:fpe}
         &\int_{\sR^d}  \psi(T,x)\nu(\d x)  
         - \int_{\sR^d} \psi(0,x)\rho(\d x)=      \int_{\sT \times \sR^d}\Big(\big(\sL^{\bm \mu,\gamma}  \psi\big)(t,x)+(\partial_t\psi)(t,x)\Big)m_t^X(\d x)\d t. 
        \end{align}
\item 
\label{item:statemargin_narrowcont_cont} 
There exists a version of $m^X$ 
that is narrowly continuous, in the sense that
for all $\phi\in C_b(\sR^d)$,
         $t\mapsto \int_{\sR^d} \phi(x) m^X_t(\d x)$ is continuous. 
         Moreover,  $m_0^X=\rho$ and $m_T^X=\nu$.

             \end{enumerate}   
\end{Proposition}

\begin{proof}

 We first prove Item \ref{item:timemargin_Les_cont}. 
    Fix $(v,\xi)\in\mathcal{D}_{{P}}(\bm \mu)$. 
   Setting $\psi=1$ in  
    \eqref{eq:condition_primal_cx_diffusion}  yields that 
    $\nu(\sR^d)=1$. 
      For all $h\in C(\sT)$,
    define  $\psi_h(t,x)=\int_0^t h(s)\d s$ for all $(t,x)\in \mathbb T\times \sR^d$.  Setting $\psi=\psi_h$ in \eqref{eq:condition_primal_cx_diffusion} 
    and noting that $\sL^{\bm \mu}\psi_h=0$ give  
    $
    \int_{\mathbb T} h(t) \xi^{\sT} (\d t)= \int_0^T h(t) \d t  
    $ for all $h\in C(\mathbb T)$.
    The Riesz–Markov–Kakutani representation theorem
    then shows that 
    $\xi^{\mathbb T}$ 
    coincides with the Lebesgue measure on $\mathcal B(\mathbb T)$.

To prove  Item \ref{item:statemargin},
 by Item \ref{item:timemargin_Les_cont} and 
 the definition of a disintegration kernel, we have for any bounded measurable function $\varphi:\sT\times\sR^d\times A\to \sR$ that
\begin{equation}\label{equ:disint_diffusion}
    \begin{aligned}
   \int_{\sT\times\sR^d\times A}\varphi(t,x,a)\xi(\d t, \d x,\d a)
   &=\int_\sT \int_{\sR^d}\int_A\varphi(t,x,a)\gamma(\d a|t,x)m_t^X(\d x)\d t.
    \end{aligned}
\end{equation}
Combining  \eqref{equ:disint_diffusion} and  \eqref{eq:condition_primal_cx_diffusion} yields 
$\int_{\sR^d}  \psi(T,x)\nu(\d x)  
 - \int_{\sR^d} \psi(0,x)\rho(\d x)
 =      \int_{\sT \times \sR^d \times A}\Big(\big(\sL^{\bm \mu}  \psi\big)(t,x,a)+(\partial_t\psi)(t,x)\Big)\gamma(\d a|t,x)m_t^X(\d x)\d t,
    \quad \forall \psi \in \cW,$
which leads to \eqref{eq:fpe}
using   the definition \eqref{eq:def_bsigmamu} of  $b^{\bm \mu, \gamma}$ and $\sigma^{\bm \mu, \gamma}$.
This proves  Item \ref{item:statemargin}. 

{
For Item \ref{item:statemargin_narrowcont_cont}, testing
\eqref{eq:fpe} with functions compactly supported in
$(0,T)\times\sR^d$ shows that $m^X$ solves
$\partial_t m_t^X=(\sL_t^{\bm\mu,\gamma})^\dagger  m_t^X$
in the distributional sense.
The coefficients $b^{\bm\mu,\gamma}$ and
$\sigma^{\bm\mu,\gamma}(\sigma^{\bm\mu,\gamma})^\top$
are bounded and Borel measurable, and hence are integrable with respect to
$m_t^X(\d x)\d t$, as $m_t^X$ is a
probability measure.
Hence \cite[Remark~2.3]{trevisan2016well} yields a unique
narrowly continuous representative
$\widetilde m:[0,T]\to\mathcal P(\sR^d)$ satisfying
$\widetilde m_t=m_t^X$ for Lebesgue-almost every $t\in \sT$.
Moreover, for every $\psi\in \cW$ and $0\le s\le t\le T$,
the following identity holds:
\begin{align}\label{eq:integralidentity}
         &\int_{\sR^d}  \psi(t,x)\widetilde m_t(\d x)  
         - \int_{\sR^d} \psi(s,x)\widetilde m_s(\d x)=      \int_s^t\int_{\sR^d}\Big(\big(\sL^{\bm \mu,\gamma}  \psi\big)(u,x)+(\partial_t\psi)(u,x)\Big)\widetilde m_u(\d x)\d u. 
        \end{align}

It remains to identify its endpoint values. Fix $\phi\in C_c^\infty(\sR^d)$ and $\eta\in C^1(\sT )$.
Apply the integrated identity \eqref{eq:integralidentity} on $\sT$ to
$\psi(t,x)=\eta(t)\phi(x)$, and subtract \eqref{eq:fpe} for the same
test function. The right-hand sides of the two identities coincide. Consequently,
\[
 \eta(T)\int_{\sR^d}\phi(x)(\widetilde m_T-\nu)(\d x) 
 -
 \eta(0)\int_{\sR^d}\phi(x)(\widetilde m_0-\rho)(\d x)=0.
\]
Taking first $\eta(t)=1-t/T$ and then $\eta(t)=t/T$
   yields $\widetilde m_0=\rho$ and $\widetilde m_T=\nu$,
   respectively, since $C_c^\infty(\sR^d)$ separates finite
   Borel measures. The flow  $\widetilde m$  thus
   satisfies the properties in  Item \ref{item:statemargin_narrowcont_cont}.}
\end{proof}

Finally, we will apply the superposition principle 
in \cite[Theorem 2.5]{trevisan2016well}
to link the narrowly continuous curve $m^X$ 
in Proposition \ref{prop:cont_margin}
to a solution of the martingale problem 
associated with the generator $\sL^{\bm \mu,\gamma} $,
and further to a weak solution to the controlled SDE \eqref{eq:state_diffusion}.

\begin{Proposition}
\label{prop:measure_diffusion}
 Suppose (H.\ref{assm:diffusion_state})
holds. Let $\bm \mu \in \cP(\sR^d|\sT)$ and $(\nu,\xi)\in  \mathcal{D}_{{P}}(\bm \mu)$.   
Let $m^X$ be the narrowly continuous weak solution of 
\eqref{eq:FP_mu_gamma} constructed in Proposition \ref{prop:cont_margin}.
Then there exists 
a probability measure $\sP\in \cP(C([0,T];\sR^d))$
  such that 
for all  $\psi \in \cW$, the process 
  $$
 t\mapsto  M^\psi_t\coloneqq \psi(t,\omega(t) )-\int_0^t 
\Big(\big(\sL^{\bm \mu,\gamma}  \psi\big)(s,\omega(s))+(\partial_t\psi)(s,\omega(s))\Big) \d s
  $$
  is a   martingale with respect to the natural filtration on
  $C([0,T];\sR^d)$, with $t\mapsto \omega(t)$ being the canonical process,
  and $\cL^\sP(\omega(t))=m^X_t$ for all $t\in [0,T]$.

  Consequently, there exists  
a process $\bm X$, defined
on some probability space $(\Omega, \cF,\sP)$, 
 such that 
$(\bm X,\gamma)\in \cA_{\rm cl}(\bm \mu)$,  
and 
$J^{\bm \mu}_{\rm cl}(\bm X,  \gamma)=J^{\bm \mu}_P(\nu,\xi) $.
\end{Proposition}

\begin{proof}[Proof of Proposition  \ref{prop:measure_diffusion}]

The existence of a probability measure 
$\sP\in \cP(C([0,T];\sR^d))$
such that $M^\psi$ is a martingale is guaranteed  
by  the superposition principle \cite[Theorem 2.5]{trevisan2016well}.
This along with   \cite[Proposition 4.11, Chapter 5]{karatzas1991brownian}
implies that 
 there exists a weak solution $\bm X$, defined on a   filtered probability space $(\Omega, \cF,\sF, \sP)$, to the   SDE
$ 
\d X_t=b^{\mu,\gamma}(t,X_t)\d t+\sigma^{\mu,\gamma}(t,X_t)\d W_t,
$ 
and satisfies $\cL^\sP(X_t)=m_t^X$ for all $t\in \sT$. 
Thus by 
Proposition \ref{prop:cont_margin}
Item \ref{item:statemargin_narrowcont_cont},
$ \cL^\sP(X_T)=m^X_T= \nu $.
Moreover,
by \eqref{equ:disint_diffusion}, 
for all
bounded measurable function $\varphi:\sT\times\sR^d\times A\to \sR$,
\begin{align*} 
  &  \int_{\sT\times\sR^d\times A}\varphi(t,x,a)\xi(\d t, \d x,\d a)
 =\int_\sT \int_{\sR^d}\int_A\varphi(t,x,a)\gamma(\d a|t,x)m_t^X(\d x)\d t
   \\
   &\quad 
   =\int_\sT \int_{\sR^d}\int_A\varphi(t,x,a)\gamma(\d a|t,x)\cL^\sP(X_t)(\d x)\d t
   =\sE^{\sP}\bigg[\int_{\sT\times A } 
   \varphi(t,X_t,a) \gamma(\d a |t,X_t) \d t\bigg].
    \end{align*}
    This implies $J^{\bm \mu}_{\rm cl}(\bm X,  \gamma)=J^{\bm \mu}_P(\nu,\xi) $
    and completes the proof. 
\end{proof}

Combining  Propositions 
\ref{prop:cl_measure} and 
\ref{prop:measure_diffusion}
shows that two formulations \eqref{controlled_cl_diffusion} and 
\eqref{Linear_programming_diffusion} of the representative player's control problem are equivalent, hence Theorem \ref{thm:primal_diffusion}.
\subsection{Proofs of
Lemma \ref{lemma:L_diffusion},
Propositions \ref{prop:D_dual_equivalence} and \ref{prop:weak_duality_diffusion},
and Corollary \ref{cor:complementary_conditions_diffusion}}  

  \begin{proof}[Proof of Lemma \ref{lemma:L_diffusion}]

  By (H.\ref{assm:diffusion_state})  and the definition of $\cW$,
  $(\psi(T,\cdot), 
  -(\partial_t \psi  + \sL^{\bm \mu}  \psi )  )\in \cY$, and hence 
  $   \cL(\nu,\xi)( \psi)$ is well defined.
  It is easy to see  
   the map $\cW\ni \psi\mapsto   \cL(\nu,\xi)( \psi)  \in \sR$ is 
   linear, and hence  
   $\cL(\nu,\xi) \in \cZ$. 
     The map $\cX\ni (\nu,\xi)\mapsto \cL(\nu,\xi)\in \cZ $ is linear by the bilinearity of $\langle \cdot,\cdot \rangle_{\cX\times \cY}$.

     To show the continuity of $\cL$, let $ (\nu^\alpha,\xi^\alpha)_{\alpha\in \Gamma} $ be a net indexed by a directed set $\Gamma$ which converges to some $(\nu,\xi)\in \cX$ in the $\sigma$-topology on $\cX$. We claim that the net 
$( \cL(\nu^\alpha,\xi^\alpha))_{\alpha\in \Gamma} $ converges to $\cL(\nu,\xi)$ in the $\sigma$-topology  on $\cZ$.
To see it, observe that 
$\cY$ and $\cW$ are the continuous  dual spaces of $\cX$  and $\cZ$ in the $\sigma$-topology, respectively; see
 \cite[Proposition 1, p.~37]{anderson1987linear}. Hence 
the  convergence of $((\nu^\alpha,\xi^\alpha))_{\alpha\in \Gamma} \in \cX$
in the $\sigma$-topology on $\cX$
 implies 
for all $(u,\phi)\in \cY$, the net
$(\langle (\nu^\alpha,\xi^\alpha), (u,\phi) \rangle_{\cX\times \cY})_ {\alpha\in \Gamma}$ converges
to $\langle  (\nu,\xi), (u,\phi) \rangle_{\cX\times \cY}$ 
 in $\sR$. 
Moreover, to show
the desired convergence of $( \cL(\nu^\alpha,\xi^\alpha))_{\alpha\in \Gamma} $ in $\cZ$, it suffices to prove that 
for all $\psi\in W$, 
the net $ ( \langle  \cL(\nu^\alpha,\xi^\alpha), \psi  \rangle_{\cZ\times \cW})_{\alpha\in \Gamma}$ converges to 
$\langle   \cL(\nu,\xi), \psi \rangle_{\cZ\times \cW} $  in $\sR$,
which is  equivalent to show the convergence of  
$\left(\left\langle 
 (\nu^\alpha,\xi^\alpha), 
 \big(\psi(T,\cdot), 
  -(\partial_t \psi  + \sL^{\bm \mu}  \psi )   \big)\right\rangle_{\cX\times \cY}\right)_{ \alpha\in \Gamma} $
to $ \left\langle 
 (\nu,\xi), \big(\psi(T,\cdot), 
  -(\partial_t \psi  + \sL^{\bm \mu}  \psi )   \big) \right\rangle_{\cX\times \cY}$, 
due to   the definitions of the bilinear form  $ \langle  \cdot, \cdot  \rangle_{\cZ\times \cW}$ and the map $\cL$. 
Note that for all $\psi\in W$, 
by (H.\ref{assm:diffusion_state}) , $\big(\psi(T,\cdot), 
  -(\partial_t \psi  + \sL^{\bm \mu}  \psi )   \big) \in \cY$, and hence the desired convergence follows from the 
convergence of $(\nu^\alpha,\xi^\alpha)_{\alpha\in \Gamma}$ in the  $\sigma$-topology on $\cX$. 
\end{proof}

 \begin{proof}[Proof of Proposition \ref{prop:D_dual_equivalence}]
 It suffices to prove that 
 the    definitions 
\eqref{eq:D_dual_diffusion_abstract}
and \eqref{eq:D_dual_diffusion}
of  
$\cD_{P^*}(\bm \mu )$ are equivalent. 
That is,
 for each $\psi\in \cW$, 
 $\psi$ satisfies \eqref{eq:D_dual_diffusion_abstract} if and only if 
 $\psi $ satisfies \eqref{eq:D_dual_diffusion}.

     Fix $\psi\in \cW$. 
    For all $(\nu,\xi)\in \cX_+$,
     by the definition of $\cL^*$ and $\cL$,
\begin{align}
\label{eq:identity_L_L^*}
\begin{split}
  &   \langle
 (\nu,\xi), 
 \big((g(\cdot,\mu_T),  f(\cdot,\cdot,\cdot,\mu_\cdot))- \cL^*(\psi)
 \big)
 \rangle_{\cX\times \cY}
 \\
 &  
 =
   \langle
 (\nu,\xi), 
  (g(\cdot,\mu_T),  f(\cdot,\cdot,\cdot,\mu_\cdot))\rangle_{\cX\times \cY} 
  - \langle
 (\nu,\xi), 
   \cL^*(\psi)
 \rangle_{\cX\times \cY}
 \\
 &  
 =
   \langle
 (\nu,\xi), 
  (g(\cdot,\mu_T)-\psi(T,\cdot),  \partial_t \psi + \sL^{\bm \mu}  \psi 
+ f(\cdot,\cdot,\cdot,\mu_\cdot))\rangle_{\cX\times \cY}.
 \end{split}
\end{align}

     Suppose that $\psi $ satisfies \eqref{eq:D_dual_diffusion}. 
By \eqref{eq:identity_L_L^*}, it holds for all
 $(\nu,\xi)\in \cX_+$,
\begin{align*}
\begin{split}
  &   \langle
 (\nu,\xi), 
 \big((g(\cdot,\mu_T),  f(\cdot,\cdot,\cdot,\mu_\cdot))- \cL^*(\psi)
 \big)
 \rangle_{\cX\times \cY}
 \\
& =
   \langle
 (\nu,\xi), 
  (g(\cdot,\mu_T)-\psi(T,\cdot),  \partial_t \psi + \sL^{\bm \mu}  \psi 
+ f(\cdot,\cdot,\cdot,\mu_\cdot))\rangle_{\cX\times \cY} 
 \ge 0,
 \end{split}
\end{align*}
where the last inequality used 
$   g(\cdot,\mu_T)\ge \psi(T,\cdot)$ and $ \partial_t \psi + \sL^{\bm \mu}  \psi 
+ f(\cdot,\cdot,\cdot,\mu_\cdot)\ge 0 $
since $\psi$ satisfies \eqref{eq:D_dual_diffusion}.
This proves $\psi$ satisfies \eqref{eq:D_dual_diffusion_abstract}.

Now suppose that $\psi$ satisfies \eqref{eq:D_dual_diffusion_abstract}.
Let $E\coloneqq \sT\times \sR^d\times  A$.
From \eqref{eq:identity_L_L^*}, it holds for all
  $(\nu,\xi)\in \cX_+$ that
  $ \langle
 (\nu,\xi), 
  (g(\cdot,\mu_T)-\psi(T,\cdot),  \partial_t \psi   + \sL^{\bm \mu}  \psi 
+ f(\cdot,\cdot,\cdot,\mu_\cdot))\rangle_{\cX\times \cY} 
 \ge 0.$
This implies that 
 for all $x\in \sR^d$,
 $g(x,\mu_T)\ge \psi(T,x)$
 and for all  $\xi \in \cM_+(E)$,
 \begin{equation}
 \label{eq:h_i_test_measure}
 \int_{E} 
 h(t,x, a)\xi(\d t,\d x, \d a)\ge 0,
 \quad \textnormal{with 
 $
 h(t,x, a)\coloneqq  \partial_t \psi (t,x) + (\sL^{\bm \mu}  \psi)(t,x,a)
+ f(t,x,a,\mu_t) $}.    
 \end{equation}
 Suppose that there exists 
 $(t_0, x_0, a_0)\in E$ such that 
 $ h (t_0,x_0,  a_0)<0$. 
 Then 
 $$\int_{E} 
 h(t, x, a)\delta_{(t_0, x_0,  a_0)} (\d t,\d x, \d a)=h(t_0,x_0,  a_0)< 0,
 $$
which contradicts     \eqref{eq:h_i_test_measure}.
Consequently,  
$ h(t,x,  a)\ge 0$ for all $(t,x,  a)\in E$.
This shows that  
 $\psi$ satisfies \eqref{eq:D_dual_diffusion} and finishes the proof.
 \end{proof}

\begin{proof}[Proof of Proposition \ref{prop:weak_duality_diffusion}]
Note that 
  if $\cD_{ P}(\bm \mu)=\emptyset$,
  then $\inf_{(\nu,\xi)\in \cD_{ P}(\bm \mu)} J_{{P}}^{\bm \mu}(\nu,\xi)=\infty$,
  and 
  if $D_{P^*}(\bm \mu)=\emptyset$,
  then $\sup_{\psi\in \mathcal D_{P^*}(\bm \mu)} J_{{P}^*}^{\bm \mu} (\psi)=-\infty$.
  Hence without loss of generality, we can assume both 
$D_{P^*}(\bm \mu)$
and $\cD_{ P}(\bm \mu)$
are non-empty. 
In this case,
the desired inequality
$\sup_{\psi\in \mathcal D_{P^*}(\bm \mu)} J_{{P}^*}^{\bm \mu} (\psi)\le \inf_{(\nu,\xi)\in \cD_{ P}(\bm \mu)} J_{{P}}^{\bm \mu}(\nu,\xi)
    $
follows from   
  the weak duality result for general infinite-dimensional LP problems \cite[Theorem 3.1, p.~39]{anderson1987linear}.
\end{proof}

\begin{proof}[Proof of Corollary \ref{cor:complementary_conditions_diffusion}]

 To show the    equivalence between \eqref{eq:NE_primal_dual_value_diffusion} and \eqref{eq:complementary_diffusion},
 observe that using $\psi^*\in \cW$ and  \eqref{eq:NE_primal_constraint_diffusion}, 
\eqref{eq:NE_primal_dual_value_diffusion} is equivalent to $
\int_{\sR^d} (g(x,\mu^*_T)- \psi^*(T,x))\mu^*_T(\d x)  
  + \int_{\sT \times \sR^d \times A}\Big(\big(\sL^{\bm \mu^*}  \psi^*\big)(t,x,a)+(\partial_t\psi^*)(t,x)+f(t,x,a,\mu^*_t)\Big)\xi^*(\d t, \d x,  \d a)=0,$
which is equivalent to \eqref{eq:complementary_diffusion} 
 under   Condition    
\eqref{eq:NE_dual_constraint_2_diffusion}.
\end{proof}

\subsection{Proofs of Theorems  
\ref{thm:primal_dual_NE_diffusion_necessary}
and \ref{thm:strong_duality_diffusion} 
}

\begin{proof}[Proof of Theorem \ref{thm:primal_dual_NE_diffusion_necessary}]
Since  $(\bm \mu^*, \bm X^*, \gamma^*)$ 
    is an MFE, $J^{\bm \mu^*}_{\rm cl}(\bm X^*,\gamma^*)=\inf_{(\bm X, \gamma)\in \cA_{\rm cl}(\bm \mu^*)} J^{\bm \mu^*}_{\rm cl}(\bm X,\gamma)$ and $\mu_T^*=\mathscr{L}^\sP (X^*_T)$.
By Proposition \ref{prop:cl_measure}, we have $(\mu_T^*,\xi^*)\in\cD_P(\bm \mu^*)$ and hence Condition \eqref{eq:NE_primal_constraint_diffusion} holds. Moreover,  $$J^{\bm \mu^*}_P(\mu_T^*,\xi^*)=J^{\bm \mu^*}_{\rm cl}(\bm X^*,  \gamma^*)=\inf_{(\bm X, \gamma)\in \cA_{\rm cl}(\bm \mu^*)} J^{\bm \mu^*}_{\rm cl}(\bm X,\gamma)=\inf_{(\nu,\xi)\in \mathcal{D}_P(\bm \mu^*) }J^{\bm \mu^*}_P(\nu,\xi),$$ 
where the last equality holds due to Theorem \ref{thm:primal_diffusion}.

Since $\bm \mu^*$ ensures strong duality, $\argmax_{\psi\in  {\cD_{P^*}}(\bm \mu^*) }   {J}^{\bm \mu^*}_{P^*}( \psi)$ is nonempty
and for any $\psi^*\in   \argmax_{\psi\in  {\cD_{P^*}}(\bm \mu^*) }   {J}^{\bm \mu^*}_{P^*}( \psi) $,
$$
J_{{P}^*}^{\bm \mu^*} (\psi^*)=
\sup_{\psi\in \mathcal D_{P^*}(\bm \mu^*)} J_{{P}^*}^{\bm \mu^*} (\psi )
=\inf_{(\nu,\xi)\in \cD_{ P}(\bm \mu^*)}J_{{P}}^{\bm \mu^*} (\nu,\xi ) 
=J^{\bm \mu^*}_P(\mu_T^*,\xi^*),
$$
which verifies Condition \eqref{eq:NE_primal_dual_value_diffusion}.  Finally, Condition
\eqref{eq:NE_dual_constraint_2_diffusion} is guaranteed by the feasibility of $\psi^*\in\cD_{P^*}(\bm \mu^*)$, and Condition \eqref{eq:NE_consistency_constraint_diffusion} is ensured by the definition of $\xi^*$. 
\end{proof}

\begin{proof}[Proof of Theorem \ref{thm:strong_duality_diffusion}]
We prove that there exists $(\bar \nu,\bar \xi)\in \cD_{P}(\bm \mu)$ such that $J_{{P}^*}^{\bm \mu} (V)=J_{{P}}^{\bm \mu}(\bar \nu,\bar\xi)$. By Assumption (H.\ref{assum:duality}\ref{item:pde_regularity}), there exists $\phi(t,x)$ such that 
    $\phi (t,x) \in \arg\min_{a\in A} \big((\sL^{\bm \mu} V)(t, x, a)+f(t,x,a,\mu_t) \big).$   
    Define $\gamma \in \cP(A|\sT\times \sR^d) $ by
    $
    \gamma(\d a|t,x):=\delta_{\phi(t,x)}(\d a).$
    By Assumption  (H.\ref{assum:duality}\ref{item:feasible}), there exists  $(\bm X,\gamma)\in \cA_{\rm cl}(\bm \mu)$ defined on some filtered probability space $(\Omega, \cF,\sF,\sP)$. Define $\bar\nu\in \cM_+(\sR^d)$ and  $\bar\xi\in \cM_+(\sT\times \sR^d\times  A)$ such that for all $F_1\in \cB(\sR^d)$, and  $F_2 \in \cB(\sT\times \sR^d\times A)$, 
    \begin{align}\label{eq:barnuxi}
    \begin{split}
    \bar\nu(F_1)&\coloneqq\sP(X_T\in F_1),
    \\ \bar\xi(F_2)&\coloneqq  \sE^{\sP}\bigg[\int_{\sT\times A } \mathds{1}_{\{(t,X_t,a)\in F_2\}} \gamma(\d a |t,X_t) \d t\bigg]=\sE^{\sP}\bigg[\int_{\sT} \mathds{1}_{\{(t,X_t,\phi(t,X_t))\in F_2\}}  \d t\bigg].
    \end{split}
    \end{align} 
    By Proposition \ref{prop:cl_measure}, $(\bar\nu,\bar\xi)\in\cD_{P}(\bm \mu)$. Then,
    \begin{equation*}
        \begin{aligned}
        J_{{P}^*}^{\bm \mu} (V)&=\int_{\sR^d}V(0,x)\rho(\d x)
        =\int_{\sR^d}  V(T,x)\bar\nu(\d x) -\int_{\sT \times \sR^d \times A}\Big(\big(\sL^{\bm \mu}  V\big)(t,x,a)+\partial_tV(t,x)\Big)\bar\xi(\d t, \d x,  \d a)
        \\
        &=\int_{\sR^d}  g(x,\mu_T)\bar\nu(\d x) -\sE^\sP\bigg[\int_{\sT}\Big(\big(\sL^{\bm \mu}  V\big)\big(t,X_t,\phi(t,X_t)\big)+\partial_tV(t,X_t)\Big)\d t\bigg]
        \\
        &=\int_{\sR^d}  g(x,\mu_T)\bar\nu(\d x)+\sE^\sP\bigg[\int_{\sT}f(t,X_t,\phi(t,X_t),\mu_t)\d t\bigg]
        \\
        &=\int_{\sR^d}  g(x,\mu_T)\bar\nu(\d x)+\int_{\sT \times \sR^d \times A}f(t,x,a,\mu_t)\bar\xi(\d t, \d x,  \d a)=J_{{P}}^{\bm \mu} (\bar\nu,\bar\xi),
        \end{aligned}
    \end{equation*}
    where the first equality holds by \eqref{eq:dual_diffusion}, the second equality holds by \eqref{eq:condition_primal_cx_diffusion}, the third equality holds by Assumption (H.\ref{assum:duality}\ref{item:pde_regularity}) and \eqref{eq:barnuxi}, the fourth equality holds by Assumption (H.\ref{assum:duality}\ref{item:pde_regularity}), the fifth equality holds by \eqref{eq:barnuxi} and the last one holds by \eqref{eq:object_primal_cX_diffusion}. This along with the weak duality implies that $(\bar\nu,\bar\xi)$ and $V$ are the optimal solutions to \eqref{Linear_programming_diffusion} and \eqref{eq:dual_diffusion}, respectively, and therefore proves the desired strong duality result.
\end{proof}

\subsection{Proof of Theorem \ref{thm:evans-krylov} }

We recall   the standard  spatial and   parabolic H\"older and Sobolev spaces, defined in \cite[Chapter II]{ladyzhenskaya1968linear}.
Specifically, for $r>0$ and $(s_0,x_0)\in\sR\times\sR^d$, define the spatial ball
$B_r(x_0)=\{x\in\sR^d \mid |x-x_0|<r\}$
and the parabolic cylinder  
$ Q_r(s_0,x_0)=(s_0-r^2,s_0+r^2)\times B_r(x_0).$
We write $B_r=B_r(0)$ and $Q_r=Q_r(0,0)$.
For $\alpha\in(0,1)$, we write 
$C_{\rm par}^{\alpha}(Q_r)=C^{\alpha/2,\alpha}(Q_r)$ and $
 C_{\rm par}^{2+\alpha}(Q_r)=C^{1+\alpha/2,2+\alpha}(Q_r)$ 
for the parabolic H\"older spaces.
The symbols $[h]_{\alpha;Q_r}$ and $\|h\|_{\alpha;Q_r}$ denote,
respectively, the parabolic H\"older seminorm and full norm, while
$\|u\|_{2+\alpha;Q_r}$ denotes the full $C_{\rm par}^{2+\alpha}(Q_r)$ norm.
For $1\le q<\infty$, $W_q^{1,2}(Q_r)$ denotes the parabolic Sobolev
space with one weak time derivative and spatial weak derivatives
up to order two in $L^q(Q_r)$.
We also recall the standard notation of viscosity solution in 
\cite[Definition 2.2]{mou2019remarks}.

We first present an interior H\"older estimate for
viscosity solutions of a parabolic HJB equation. 

 \begin{Proposition}
 \label{prop:local_hjb_regularity}
 Suppose that $(\Sigma^a,b^a,f^a)_{a\in A}$ satisfy the ellipticity and H\"older
norm bounds in Theorem \ref{thm:evans-krylov}, with $S_T$ replaced by $Q_2$. 
Let $V\in C(Q_2)\cap L^\infty(Q_2)$ be a viscosity solution  to 
\begin{equation}
 \partial_s V(s,x)+F(s,x,D_xV(s,x),D^2_xV(s,x))=0\quad\textnormal{in $Q_2$},
 \label{eq:localPDE}
\end{equation}
There exists  $\alpha_0=\alpha_0(d,\kappa,K)\in(0,1)$ such that  for all 
$0<\alpha<\min\{\alpha_0,\eta\}$,
$V\in C_{\rm par}^{2+\alpha}(Q_1)$ and 
$
 \|V\|_{2+\alpha;Q_1}
 \le C\left(\|V\|_{\infty;Q_2}
                    +\sup_{a\in A}\|f^a\|_{\alpha;Q_2}\right)
$,
where the    constant  $C$ depends on $\alpha, \eta, d,\kappa,K, L$, but is   independent of $L_f$.

 \end{Proposition} 

\begin{proof}
 Let $\mathcal N \coloneqq 
 \sup_{a\in A} \|f^a\|_{\alpha;Q_2}+ \|V\|_{\infty;Q_2}$.
 For any $\alpha\in (0,1)$, set 
$ q= {2(d+2)}/{(1-\alpha)}$,
and 
$\theta= (1+3\alpha)/4$, so that 
$q>{d+1}$ and $\alpha<\theta<1-(d+2)/q$.
Applying 
\cite[Theorem 5.1]{mou2019remarks}  (with the outer domain $Q_2$ and the inner domain $Q_{3/2}$) yields
$
 \|V\|_{W_q^{1,2}(Q_{3/2})}\le C\mathcal N
$. 
The parabolic Sobolev embedding \cite[Lemma 3.3, p.~80]{ladyzhenskaya1968linear} along with   $\theta<1-(d+2)/q$    implies that 
$ \|{D_x V}\|_{\theta;Q_{5/4}}\le C\|V\|_{W_q^{1,2}(Q_{3/2})}\le C\mathcal N.
$
On $Q_{5/4}$, define $\ell^a=(b^a)^\top D_x V+f^a$. Then for all $a\in A$, 
since $\theta>\alpha$, 
\begin{align*}
 \|{\ell^a}\|_{\alpha;Q_{5/4}}
 &\le C\bigl(\|{b^a}\|_{\infty;Q_{5/4}}
                      \|{D_x V}\|_{\alpha;Q_{5/4}}
       +[b^a]_{\alpha;Q_{5/4}}\|{D_x V}\|_{\infty;Q_{5/4}}
       +\|{f^a}\|_{\alpha;Q_{5/4}}\bigr)\\
 &\le C( \|{D_x V}\|_{\theta;Q_{5/4}}+ \|{f^a}\|_{\alpha;Q_{5/4}})\le C\mathcal N,
\end{align*}
Since $V$ is differentiable on $Q_{5/4}$, it holds  that  $V$ is a viscosity solution to 
$$ 
 \partial_t V(t,x)+ \inf_{a\in A} \left\{\tfrac12\tr(\Sigma^a(t,x) D^2_x V(t,x))+\ell^a(t,x)\right\}=0,  
 \quad\textnormal{in $Q_{5/4}$}.
$$
Then by \cite[Theorem 5.2]{mou2019remarks}, there exists $\alpha_0\in (0,1)$ such that for all $\alpha\in (0,\alpha_0)$, 
\[
 \|V\|_{2+\alpha;Q_1}
 \le C\left(\|V\|_{\infty;Q_{5/4}}
                   +\sup_{a\in A}\|\ell^a\|_{\alpha;Q_{5/4}}\right)
 \le C\mathcal N.
\]
The exponent $\alpha_0$ depends only on $\kappa$, $K$, and $d$; see
Theorem 3.1 and the discussion immediately following Remark 3.2
in \cite{goffi2026interpolated}.
This finishes the proof. 
 \end{proof}

 The following estimate for  the mollifier will be used in the analysis. 
 \begin{Lemma}\label{lem:scaled_convolution}
Let $\eta\in(0,1)$, $h\in C_b^\eta(\sR^d)$, and
$K\in C_c^\infty(\sR^d)$. Define 
$
 U_K(s,x)=\int_{\sR^d}K(z)h(x-\sqrt{s} z) dz 
$
 for all $(s,x)\in[0,1]\times\sR^d$.
Then
 $U_K$ is smooth for $s>0$, and  
 for all   $s\in (0,1]$, 
 $
 \|D_xU_K(s,\cdot)\|_{\infty}
 \le Cs^{(\eta-1)/2}[h]_\eta$ and 
$
 \|\partial_sU_K(s,\cdot)\|_\infty
 \le Cs^{\eta/2-1}[h]_\eta.
$
Moreover, 
 the map $(s,x)\mapsto 
 s\partial_sU_K(s,x)$
admits a  continuous extension to \([0,1]\times\mathbb R^d\), which has  value zero on \(\{0\}\times\mathbb R^d\),
and  satisfies 
$
 \|U_K\|_{C_b^{\eta/2,\eta}([0,1]\times\sR^d)}
 +\|s\partial_sU_K\|_{C_b^{\eta/2,\eta}
                         ([0,1]\times\sR^d)}
 \le C\|h\|_{C_b^\eta(\sR^d)}.
$
Here  $C>0$ depends only on $d,\eta$, and $K$.

\end{Lemma}

\begin{proof}

We write $*$ for convolution in the spatial variables and 
define for all $\varrho\in C^\infty_c(\sR^d)$
and $r>0$,   $\varrho_r(y)\coloneqq r^{-d}\varrho(y/r)$.
For $(s,x),(s',x')\in[0,1]\times\sR^d$,  
\begin{equation}
 \label{eq:mollifier_holder}
 \begin{aligned}
 |U_K(s,x)-U_K(s',x')|
 &\le [h]_\eta\int_{\sR^d}|K(z)|
 \bigl|x-x'-(\sqrt{s}-\sqrt{s'})z\bigr|^\eta\,\d z\\
 &\le C[h]_\eta
       \bigl(|x-x'|^\eta+|s-s'|^{\eta/2}\bigr).
 \end{aligned}
\end{equation}
Together with
$\|U_K\|_\infty\le\|K\|_{L^1}\|h\|_\infty$,
this proves the     H\"older norm bound for $U_K$.

For $s>0$, the identity $U_K(s,\cdot)=h*K_{\sqrt{s}}$
shows that $U_K$ is smooth.
Writing $r=\sqrt{s}$ and using $\int_{\sR^d}DK(x)\d x=0$, we obtain
$D_xU_K(s,x)
=r^{-1}\int_{\sR^d}DK(z)\bigl(h(x-rz)-h(x)\bigr)\,\d z$,
which along with 
the H\"older continuity of $h$   gives
$\|D_xU_K(s,\cdot)\|_\infty
\le Cr^{\eta-1}[h]_\eta
=Cs^{(\eta-1)/2}[h]_\eta$.

Finally, 
define
$\psi_K(z)\coloneqq-dK(z)-z^\top DK(z)
$ for all $z\in \sR^d$,
and  $U_{\psi_K}(s,x)= (h*(\psi_K)_{\sqrt{s}})(x)$ for all $(s,x)\in (0,1]\times \sR^d$.
For all $s>0$,
a direct computation gives
$\partial_sK_{\sqrt{s}}(y)
=(2s)^{-1}(\psi_K)_{\sqrt{s}}(y)$,
and hence
 $ \partial_sU_K(s,x)
 =\frac{1}{2s}(h*(\psi_K)_{\sqrt{s}})(x)=\frac{1}{2s} U_{\psi_K}(s,x)$.
By the integration-by-parts formula, 
$\int_{\sR^d}\psi_K(z) \d z=0$, and similar to  
 $D_x U_K$, 
$\|\partial_sU_K(s,\cdot)\|_\infty
\le Cs^{\eta/2-1}[h]_\eta$.
Applying \eqref{eq:mollifier_holder} and the
sup-norm estimate with $K$ replaced by $\psi_K$ proves
the   H\"older norm bound for $U_{\psi_K}$.
Since $U_{\psi_K}(0,x)=h(x)\int_{\sR^d}\psi_K(z)\,dz=0$,
the identity $s\partial_sU_K=\tfrac12U_{\psi_K}$ provides
the required continuous extension, with value zero at $s=0$.
This finishes the proof. 
  \end{proof}

Now we are ready to prove Theorem \ref{thm:evans-krylov}. 

\begin{proof}[Proof of Theorem  \ref{thm:evans-krylov}. ]

The proof   establishes the existence of a classical solution to \eqref{eq:hjb_holder} in Steps~1--3, and then the  uniqueness in Step 4.
To simplify   notation, 
let    $C$ denote a generic constant that may vary from line to line and depends only on the data   in the statement. In particular, $C$ is independent of the regularization parameter and the spatial centre of the cylinders.
When a norm is taken over the whole space, we omit the domain from its notation. 

\textbf{Step 1.}
 We first construct a sequence of approximate solutions to \eqref{eq:hjb_holder} with smooth coefficients.  
For each $h\in \{\Sigma^a,b^a,f^a\}_{a\in A}$, 
let $h_{\rm ext}$ be the extension of $h$ onto $\sR\times \sR^d$ given by  
$h_{\rm ext}(t,x)=h(\min\{T,\max\{0,t\}\},x)$. 
To mollify the coefficients, let 
$\chi:\sR\to [0,\infty)$ and $\varrho:\sR^d\to [0,\infty)$ be 
  smooth
  functions of  integral one, 
   supported on $(-1,1)$ and $B_1$, respectively.  
For each   $\eps \in (0,1)$, define the mollifiers 
$ \varrho_\eps(x)\coloneqq \eps^{-d}\varrho(x/\eps)$ and $
 \zeta_\eps(t,x)\coloneqq \eps^{-d-2}\chi(t/\eps^2)\varrho(x/\eps)
$, 
and 
the  mollified coefficients  by 
$$
 h_\eps (t,x)
  =\int_{\sR^{1+d}}\zeta_\eps(r,y)h_{\rm ext} (t-r,x-y)\,dr\,dy,  
 \quad 
 g_\eps(x) = (\varrho_\eps*g)(x), 
$$
where $h\in \{\Sigma^a,b^a,f^a\}_{a\in A}$ and the symbol $*$ denotes the spatial convolution.
   Define the Bellman operators $F_{\rm ext}$ 
   and 
$F_\eps$ as in \eqref{eq:hamiltonian_regularity}, 
using the extended   and   mollified coefficients, respectively. 
The H\"older continuity of the coefficients
and standard estimates of   mollifiers (see e.g., \cite[p.~84]{ladyzhenskaya1968linear})
 ensure that 
 for all $\eps \in (0,1)$ and $h\in \{\Sigma^a,b^a,f^a\}_{a\in A}$, 
 $
 \|h_\eps  \|_{C_b^{\eta/2,\eta}}
       \le C  
$,
$
 \|h_\eps-h_{\rm ext}\|_\infty
       \le C\eps^\eta 
$, 
 $
\|g_\eps  \|_{C_b^{2+\eta}}
       \le C  
$,
and 
$ \|g_\eps-g\|_{C_b^2}
       \le C\eps^\eta $.
Thus, by \cite[Remark 2 and Theorem 4]{diehl2014regularity},
for all $\eps \in (0,1)$, 
there exists a unique 
  bounded  
continuous viscosity solution $V_\eps$ to the following HJB equation: 
\begin{equation}
\label{eq:hjb_moliffied}
 \partial_t V_\eps(t,x) +F_\eps(t,x,D_x V_\eps(t,x),D^2_x V_\eps(t,x))=0,
 \quad (-1,T)\times\sR^d; 
 \quad V_\eps(T,\cdot)=g_\eps,
\end{equation}
and $\|V_\eps\|_{\infty;[-1,T]\times\sR^d}\le \|g\|_\infty+(T+1)L_f $.  

\textbf{Step 2.} To estimate the regularity of 
 $V_\eps$ using  Proposition 
\ref{prop:local_hjb_regularity}, we then  
construct an extension $\widetilde V_\eps$ of  $V_\eps$ 
on 
$[-1,T+1]\times\sR^d$, which is a  viscosity solution to   the following     HJB   equation: 
\begin{equation}
\label{eq:extended_hjb}
\partial_t \widetilde V_\eps(t,x)+
 \widetilde F_\eps(t,x,D_x \widetilde V_\eps(t,x),D^2_x \widetilde V_\eps(t,x))=0
 \quad\text{in }(-1,T+1)\times\sR^d,
\end{equation}
where $ \widetilde F_\eps$ is an extension of $F_\eps$   specified later. 
To this end, fix $\eps\in (0,1)$. Define $h_\eps:\sR^d\to \sR$ by 
$
 h_\eps(x)=F_\eps(T,x,Dg_\eps(x),D^2  g_\eps(x)),
$
and  $w_\eps:[T,T+1]\times \sR^d\to \sR$ by 
\[
 w_\eps(t,x)=g_\eps(x)
       -(t-T)H_\eps( {\sqrt{t-T}}, x),
 \quad T<t\le T+1;
 \quad w_\eps(T,x)=g_\eps(x).
\]
where $H_\eps(r,x)\coloneqq  (h_\eps*\varrho_{r})(x)$.
We claim that 
\begin{equation}
 \label{eq:Aw}
 \|w_\eps\|_{2+\eta;[T,T+1]\times\sR^d}\le C,
 \quad
 \partial_t w_\eps (T,\cdot)=-h_\eps,\quad
 D_xw_\eps(T,\cdot)=Dg_\eps, 
 \quad 
 D^2_x w_\eps(T,\cdot)=D^2  g_\eps.
\end{equation}
Indeed,  
the coefficient regularity   implies
$\|h_\eps\|_{C_b^\eta }
+\|g_\eps\|_{C_b^{2+\eta}}\le C$.
For each   $K\in C^\infty_c(\sR^d) $, let $U_{K,\eps}(s,\cdot)=h_\eps*K_{\sqrt{s}}$ for all $s>0$.
Thus, writing $s=t-T$, we have
$H_\eps(\sqrt{s},x)=U_{\varrho,\eps}(s,x)$ and
$w_\eps(t,x)=g_\eps(x)-sU_{\varrho,\eps}(s,x)$.
For all $t>T$, 
and $1\le i,j \le d$,
$ \partial_t w_\eps(t,\cdot)
 =-U_{\varrho,\eps}(s,\cdot)
   -s \partial_sU_{ \varrho,\eps}(s,\cdot)$,
$    \partial_{x_i}w_\eps(t,\cdot)
=\partial_{x_i}g_\eps-\sqrt{s}\,U_{\partial_{x_i}\varrho,\eps}(s,\cdot) $,
   and   
   $ \partial_{x_ix_j} w_\eps(t,\cdot)
 =\partial_{x_ix_j} g_\eps
   -U_{\partial_{x_ix_j} \varrho,\eps}(s,\cdot)$.
By    Lemma~\ref{lem:scaled_convolution},
$ \partial_t w_\eps$,
$D_x w_\eps(t,\cdot)$ 
and $D^2_x w_\eps(t,\cdot)$ 
extend continuously   to $t=T$
and satisfy 
$ \|w_\eps\|_{2+\eta;[T,T+1]\times\sR^d}\le C$ and 
$\partial_t w_\eps(T,\cdot)= -U_{\varrho,\eps}(0,\cdot)=-h_\eps$. 

 Given $w_\eps$, define 
 $R_\eps:[-1,T+1]\times\sR^d\to \sR$ by 
 $ R_\eps(t,x)\coloneqq -\partial_t w_\eps(t,x) -F_\eps(T,x,D_x w_\eps(t,x),D^2_x w_\eps(t,x))$ for all $t>T$
 and $R_\eps(t,x) \coloneqq 0$ for all $t\in [-1,T]$. 
 Note that  
 $\|R_\eps\|_{C_b^{\eta/2,\eta}}\le C$ 
 by \eqref{eq:Aw}. 
 Define
\[
 \widetilde V_\eps(t,x)=
 \begin{cases}
 V_\eps(t,x),&t\in [-1, T], \\
 w_\eps(t,x),&t\in (T,  T+1],
 \end{cases}
 \quad 
 \widetilde F_\eps(t,x,p,M)=
 F_\eps(\pi(t),x,p,M)+R_\eps(t,x),
\]
where $\pi(t)\coloneqq\min\{t,T\}$ for $t\in[-1,T+1]$.
Note that $V_\eps(T,\cdot)=w_\eps(T,\cdot)=g_\eps$ 
and by 
the definition of $R_\eps$, 
  $\partial_t w_\eps+ \widetilde F_\eps(t,x,D_x w_\eps,D^2_x w_\eps)=0$ for all $t>T$. 
Thus one can verify from the definition that   $\widetilde V_\eps$ is a continuous viscosity solution to \eqref{eq:extended_hjb}.
Note that 
$\widetilde F_\eps$ is the Bellman operator 
given in \eqref{eq:hamiltonian_regularity}, 
with coefficients 
$\widetilde\Sigma_\eps^a(t,x)
\coloneqq\Sigma_\eps^a(\pi(t),x)$,
$\widetilde b_\eps^a(t,x)
\coloneqq b_\eps^a(\pi(t),x)$,
and
$\widetilde f_\eps^a(t,x)
\coloneqq f_\eps^a(\pi(t),x)+R_\eps(t,x)$. 

\textbf{Step 3.}
 Fix  $\alpha\in(0,\min\{\alpha_0,\eta\})$ with $\alpha_0\in (0,1)$ in  
Proposition~\ref{prop:local_hjb_regularity}, and set $r_0=1/4$.
For every $(t_0,x_0)\in S_T$,  
$Q_{2r_0}(t_0,x_0)\subset $ 
$(-1,T+1)\times\sR^d$.
 Define the bijection
$\Phi: Q_2\to Q_{2r_0}(t_0,x_0)$ such that 
$\Phi(\tau,y)\coloneqq(t_0+r_0^2\tau,x_0+r_0y)$,
which also maps  
 $Q_1$ onto $Q_{r_0}(t_0,x_0)$.
Define the rescaled function 
  $\widehat V_\eps\coloneqq\widetilde V_\eps\circ\Phi$.
  Since  $\widetilde V_\eps$ is a continuous viscosity solution to \eqref{eq:extended_hjb}, 
it is easy to show that  $\widehat V_\eps$ 
is a viscosity
solution of
$\partial_\tau\widehat V_\eps+
\widehat F_\eps(\tau,y,D_y\widehat V_\eps,
D_y^2\widehat V_\eps)=0$ in $Q_2$, where
\begin{equation}
 \label{eq:scaled_operator}
 \begin{aligned}
 \widehat F_\eps(\tau,y,p,M)
 &\coloneqq r_0^2\widetilde F_\eps
 \left(\Phi(\tau,y),\frac{p}{r_0},\frac{M}{r_0^2}\right)\\
 &=\inf_{a\in A}
 \left\{
 \tfrac12\operatorname{tr}
       \bigl(\widehat\Sigma_\eps^a(\tau,y)M\bigr)
 +\widehat b_\eps^a(\tau,y)^\top p
 +\widehat f_\eps^a(\tau,y)
 \right\}.
 \end{aligned}
\end{equation}
Here
$\widehat\Sigma_\eps^a=\widetilde\Sigma_\eps^a\circ\Phi$,
$\widehat b_\eps^a=r_0\widetilde b_\eps^a\circ\Phi$,
and
$\widehat f_\eps^a=r_0^2\widetilde f_\eps^a\circ\Phi$.
The   rescaling preserves the ellipticity bounds 
$\kappa I_d\preceq\widehat\Sigma_\eps^a
\preceq KI_d$, and scales 
the H\"older seminorms   according to
$[q\circ\Phi]_{\eta;Q_2}
=r_0^\eta[q]_{\eta;Q_{2r_0}(t_0,x_0)}$.
Thus, the coefficients in \eqref{eq:scaled_operator} have H\"older norms uniformly in $\eps$, 
and  
by Proposition~\ref{prop:local_hjb_regularity},  
$$\|\widehat V_\eps\|_{2+\alpha;Q_1}
\le C\left(\|\widehat V_\eps\|_{\infty;Q_2}
+\sup_{a\in A}\|\widehat f_\eps^a\|_{\alpha;Q_2}\right)
\le C,$$ 
where the last inequality  used $\|V_\eps\|_{\infty;[-1,T]\times\sR^d}\le C$.  
Rescaling back yields
$ \|\widetilde V_\eps\|_{2+\alpha;
             Q_{r_0}(t_0,x_0)}
 \le C$ for all $ (t_0,x_0)\in S_T$.
These  uniform local estimates can then be passed to the global
norm $\|V_\eps\|_{2+\alpha;S_T}\le C$ by a standard localization argument (see e.g.,  the proof of 
\cite[Chapter IV,    Lemma 4.2,
pp.~302--303]{ladyzhenskaya1968linear}).
Indeed, nearby points lie in a common cylinder on which
$\|\widetilde V_\eps\|_{2+\alpha;
             Q_{r_0}(t_0,x_0)}
 \le C$ applies, whereas the H\"older quotients
for points separated by the fixed (parabolic) distance $r_0$ are
controlled by the uniform sup-norm bounds.

By the Arzel\`a--Ascoli theorem
and a diagonal argument, there exists 
a sequence $(\eps_n)_{n\in \sN}$ with $\lim_{n\to \infty}\eps_n=0$ and a function
$V\in C_{\rm par}^{2+\alpha}(S_T)$ such that 
 $(V_{\eps_n}, \partial_t V_{\eps_n}, D_x V_{\eps_n}, D_x^2V_{\eps_n})_{n\in \sN}$
 converge locally uniformly
on $S_T$ to $(V, \partial_t V,  D_xV ,  D_x^2V)$. 
 The uniform convergence of the mollified coefficients,
uniformly over $a\in A$, implies that
$(F_{\eps_n})_{n\in \sN}$ converges to  $F$ uniformly on compact subsets
of $S_T\times \sR^d\times \sS^d$. 
Using the uniform bounds of the derivatives and 
passing $n\to \infty$ in  \eqref{eq:hjb_moliffied} 
yields  
$\partial_tV+F(t,x,D_xV,D_x^2V)=0$
on $(0,T)\times\sR^d$.
The fact that 
  $V_{\eps_n}(T,\cdot)=g_{\eps_n}$ and  
$\lim_{n\to\infty}\|g_{\eps_n}- g\|_\infty=0$ shows $V(T,\cdot)=g$.
This proves  that $V$ is a classical solution to \eqref{eq:hjb_holder}.

\textbf{Step 4.}
To prove the uniqueness, let $U,V\in C^{1,2}_b(\sT\times \sR^d)$ be two classical solutions to \eqref{eq:hjb_holder}.
Define $B=\sup_{ a\in A}\|b^a \|_\infty$, 
$\Lambda=\sup_{ a\in A}\|\tr\Sigma^a\|_{\infty}$, and
$L^a\phi=\frac12\tr(\Sigma^aD^2\phi)+(b^a)^\top D\phi$
for $\phi\in C^{1,2}(S_T)$. Define
$k=B+\Lambda+1$, and 
$h(t,x)=e^{k(T-t)}(1+|x|^2)$
for all $(t,x)\in S_T$. Then for all $(t,x)\in S_T$, by   $2|x|\le 1+|x|^2$,
\begin{equation}\label{eq:barrier}
 \begin{aligned}
 \partial_t h(t,x)+\sup_{a\in A}L^ah(t,x)
 &\le e^{\kappa(T-t)}      \bigl[-\kappa(1+|x|^2)+\Lambda+2B|x|\bigr] 
  \\
  &
  \le-(k-B-\Lambda)h(t,x)=-h(t,x).
 \end{aligned}
\end{equation}

Fix $\varepsilon>0$, $\delta\in (0,T)$, and let $V_\varepsilon=V+\varepsilon h$.
Observe that  for any two  bounded families $(\alpha_a)_{a\in A},  (\beta_a)_{a\in A}\subset \sR$,
$\inf_{a\in A}(\alpha_a+\beta_a)\le\inf_{a\in A}\alpha_a+\sup_{a\in A} \beta_a$. 
This with the fact that $V$ 
satisfies \eqref{eq:hjb_holder} yields
\begin{equation}\label{eq:strict}
 \begin{aligned}
 &\partial_t V_\varepsilon(t,x)+
 F(t,x,DV_\varepsilon(t,x),D^2V_\varepsilon(t,x))\\
 &\qquad\le
 \varepsilon\left(\partial_t h(t,x)+\sup_{a\in A}L^ah(t,x)\right)
 \le-\varepsilon h(t,x) <0.
 \end{aligned}
\end{equation}
where the second inequality used  \eqref{eq:barrier}.
Let $M_0=\sup_{S_T}(U-V)^+<\infty$, and choose
$R_\varepsilon>0$ so that
$\varepsilon(1+R_\varepsilon^2)>M_0$.
For every $R\ge R_\varepsilon$, if $(t,x)\in  \mathbb T\times\partial B_R$,  
$
(U-V_\varepsilon)(t,x)
\le M_0-\varepsilon(1+R^2)<0
$, and if $(t,x)\in \{T\}\times \sR^d$,  the terminal conditions of $U$ and $V$ gives
$U(T,x)-V_\varepsilon(T,x)=-\varepsilon(1+|x|^2)<0$.
Since $U$ satisfies 
\eqref{eq:hjb_holder},
applying the classical comparison principle
\cite[Theorem 14.1]{lieberman1996second} on $[\delta,T]\times B_R$ (with a  time reversal  to match the
forward-time convention therein) gives 
$U-V_\eps\le 0$
  on $[\delta,T]\times B_R$ for all $R\ge R_\eps$, which implies that $U-V_\eps\le 0$
  on $[\delta,T]\times \sR^d$.
  Using the continuity of $U-V_\eps$ and then sending $\eps\to 0$ yield  that $U\le V  $ on $S_T$. 
Interchanging the roles of $U$ and $V$ in the above analysis gives $V\le U$. This proves $U=V$ and hence the uniqueness of classical solutions. 
\end{proof} 

\subsection{Proof of 
Proposition \ref{prop:pde_regularity}}

\begin{proof}[Proof of Proposition \ref{prop:pde_regularity}]
Fix $\bm \mu\in \cU_\rho$.  It suffices to prove the existence of a classical solution
  $V\in C^{1,2}_b(\sT\times \sR^d)$ to \eqref{eq:hjb_diffusion}, since once such a  $V$ is obtained, 
  the measurable selection theorem \cite[D.5 Proposition]{hernandez2012discrete}
  ensures the 
   existence of   a measurable pointwise minimizer $\phi$.
   The associated controlled state process  admits a weak solution 
due to the non-degeneracy of  the diffusion coefficient  
and 
\cite[Theorem 6.13, p.~46]{yong2012stochastic}.

 Write  $h^{\bm\mu} (t,x,a)\coloneqq h(t,x,a,\mu_t)  $ for all $h\in \{b,\sigma,f\}$, 
 $g^{\bm \mu}(x) \coloneqq  g(x,\mu_T)$, and  $
\Sigma^{\bm \mu  }=\sigma^{\bm \mu  }(\sigma^{\bm \mu  })^\top $.
Since $\bm \mu\in \cU_\rho$, by \eqref{eq:admissible_flow} and the boundedness of $\bar b$ and 
  $\bar \sigma$, there exists $C\ge 0$ such that
$$
 \sE^\sP[|X_t-X_s|]
\le 
\sE^\sP\left[\left|\int_s^t \bar b_r \d r\right|+
 \left|\int_s^t \bar \sigma_r \d W_r \right|\right]\le C|t-s|^{1/2}, \quad  \forall t,s\in \sT.  
 $$
The definition of $\d_{\rm BL}$ implies that 
   $\d_{\rm BL}(\mu_t,\mu_s)\le  \sE^\sP[|X_t-X_s|] \le  C|t-s|^{1/2}$
 for all $t,s\in \sT$.
This along with the H\"{o}lder regularity of $b$, $\sigma$ and $f$ implies that 
for all $a\in A$,
the components of 
$b^{\bm \mu}(\cdot,a)$, 
$\Sigma^{\bm \mu}(\cdot,a)$
and $f^{\bm \mu}(\cdot,a)$
are in $C^{\eta/2,\eta}_b(\sT\times \sR^d)$ for some $\eta\in (0,1]$,
and the $|\cdot|_{{\eta/2,\eta}}$-norms
are  uniformly bounded with respect to $a\in A$.  
Hence by Theorem \ref{thm:evans-krylov}, 
the HJB equation \eqref{eq:hjb_diffusion} admits a classical solution $V$.
 This completes the proof.
\end{proof}

 \subsection{Proofs of 
Theorem  \ref{thm:HJB-FP-PD} and Examples \ref{ex:discrete_action} and \ref{ex:nonsmooth_value} }

\begin{proof}[Proof of Theorem \ref{thm:HJB-FP-PD}]
 Observe 
from Theorem \ref{thm:primal_dual_NE_diffusion_sufficient} and Corollary 
\ref{cor:complementary_conditions_diffusion}
that for 
 any  $(\bm \mu^*, \psi^*)$ satisfying
\eqref{eq:comparision_FP_HJB},
by setting 
  $\xi^*(\d t,\d x,\d a)=\delta_{\phi(t,x)}(\d a)\mu_t^*(\d x)\d t$,
the triple 
$(\bm 
\mu^*, \xi^*,\psi^*)$
satisfies the primal-dual system. Hence it suffices to prove  
 that any solution $(\bm 
\mu^*, \xi^*,\psi^*)$
to  the primal-dual system  satisfies the HJB-FP system. 

We first show 
 that $\psi^*$ satisfies the HJB equation \eqref{eq:comparison_HJB}. That is,
for all $(t,x)\in \sT\times\sR^d$, 
$\psi^*(T,x)= g(x, \mu^*_T)$
and $\min_{a\in A}F(t,x,a)=0$ with 
$$F(t,x,a)\coloneqq  \partial_t \psi^*(t,x)+ H\big(t,x, D_x \psi^* (t,x), D^2_x \psi^* (t,x),a, \mu^*_t  \big).
$$
Since $F(t,x,a)\geq 0$ by \eqref{eq:NE_dual_constraint_2_diffusion}, suppose $\min_{a\in A}F(t_0,x_0,a)=\delta>0$ for some $(t_0,x_0)\in (0,T]\times \sR^d$. 
By the continuity of the mapping $(t,x)\mapsto \min_{a\in A} F(t,x,a)$, 
there exist neighborhood $U(t_0)$ of $t_0$ and neighborhood $\sB_r(x_0)$ of $x_0$ such that $\min_{a\in A}F(t,x,a)\geq\delta/2$ for any $(t,x)\in U(t_0)\times \sB_r(x_0)$. Then 
$$
\int_{\sT\times \sR^d\times A}F(t,x,a)\xi^*(\d t,\d x,\d a)\geq \int_{U(t_0)\times \sB_r(x_0)\times A}\frac{\delta}{2}\xi^*(\d t,\d x,\d a)=\frac{\delta}{2}\int_{U(t_0)}\mu_t^*(\sB_r(x_0))\d t>0,
$$
where the   equality holds due to \eqref{eq:NE_consistency_constraint_diffusion}. This contradicts   \eqref{eq:complementary_diffusion}, and yields     $\min_{a\in A}F(t,x,a)=0$ for all $(t,x)\in \sT\times \sR^d$. Similar argument shows  $\psi^*(T,x)= g(x, \mu^*_T)$.

We then prove that $\bm \mu^*$ satisfies the FP equation
\eqref{eq:comparison_FP}.
By Proposition \ref{prop:cont_margin}  and \eqref{eq:NE_consistency_constraint_diffusion},
there exists a 
$\mu_t^*(\d x)\d t$-a.e.~unique  $\gamma \in \cP(A|\sT\times \sR^d)$ such that 
 $\xi^*(\d t,\d x,\d a)=\gamma(\d a|t,x)  \mu_t^*(\d x)\d t$.
 We claim that 
 $\gamma(\d a|t,x)  =\delta_{\phi(t,x)}(\d a)$ for $\mu_t^*(\d x)\d t$-a.e.~$(t,x)\in \sT\times \sR^d$.

 Indeed, 
\eqref{eq:NE_dual_constraint_2_diffusion}
and
\eqref{eq:complementary_diffusion}
imply that 
 $\int_A F(t,x,a) \gamma (\d a|t,x)=0$
 for $\mu_t^*(\d x)\d t$-a.e.~$(t,x)\in \sT\times \sR^d$.
Fix such a   $(t,x)\in \sT\times \sR^d$.
Since  $\phi(t,x)$ is the unique minimizer of $F(t,x,a)$,
and $\min_{a\in A} F(t,x,a)=0$,
it holds that  $F(t,x,a)>0$ for any $a\not =\phi(t,x)$. 
This implies that $\gamma(\{a \not =\phi(t,x)\}|t,x)=0$,
since otherwise 
$
\int_A F(t,x,a) \gamma (\d a|t,x)
\ge 
\int_{\{a  \not = \phi(t,x)\}} F(t,x,a) \gamma (\d a|t,x)>0.
$ 
This implies that $\gamma (\cdot|t,x)=\delta_{\phi(t,x)}$
for $\mu_t^*(\d x)\d t$-a.e.~and hence $\xi^*(\d t,\d x,\d a)=\delta_{\phi(t,x)}(\d a)\mu_t^*(\d x)\d t$.  Substituting this form of $\xi^*$ into \eqref{eq:NE_primal_constraint_diffusion}   yields \eqref{eq:comparison_FP}. This finishes the proof. 
\end{proof}

\begin{proof}[Proof of Example \ref{ex:discrete_action}]
For any admissible policy $\gamma$ and a given   flow $\boldsymbol{\mu}$, 
the objective is
$$
J^{\boldsymbol{\mu}}_{\rm cl}(\bm X, \gamma)=-\kappa\int_0^T M(\mu_t)\mathbb E[h(X_t^\gamma)]\,\d t-\lambda M(\mu_T)\mathbb E[h(X_T^\gamma)].
$$
By one-dimensional comparison, for every $t\in[0,T]$,
$\mathbb E[h(X_t^-)]\le\mathbb E[h(X_t^\gamma)]\le\mathbb E[h(X_t^+)]$,
where the bounds are attained by the constant policies $\phi^-\equiv-1$ and $\phi^+\equiv1$, respectively. Since $h$ is odd and strictly increasing, $M(\mu_t^+)>0$ and $M(\mu_t^-)<0$ for all $t>0$. Consequently, $\phi^+$ is optimal given $\boldsymbol{\mu}^+$, while $\phi^-$ is optimal given $\boldsymbol{\mu}^-$. Thus, both policies induce strict MFEs.

For the relaxed policy $\bar  \gamma $, the averaged drift is zero, so its induced flow is $\bar  \mu_t =\mathcal N(0,\sigma_0^2t)$. Since $h$ is odd, $M(\bar \mu_t)=0$ for all $t$, and hence $J^{\boldsymbol{\bar \mu} }(\gamma)=0$ for every admissible policy $\gamma$. Therefore, $\bar \gamma $ is optimal given $\boldsymbol{\bar  \mu} $ and induces a third MFE.
Now 
  suppose that a strict   policy $\phi:\mathbb T\times\sR\to A$ induces the flow $\boldsymbol{\bar \mu}$. Since $\bar  \mu_t =\mathcal N(0,\sigma_0^2t)$, its density $p_t$ is strictly positive for every $t>0$ and satisfies the heat equation $\partial_t p_t=\frac{\sigma_0^2}{2}\partial_{xx}p_t$. Moreover, the FP equation \eqref{eq:FP}   under $\phi$  reduces to 
$\partial_t p_t=-\partial_x\big(\phi(t,\cdot)p_t\big)+\frac{\sigma_0^2}{2}\partial_{xx}p_t$.
Comparing the two equations yields $\partial_x\big(\phi(t,\cdot)p_t\big)=0$ in the distributional sense for almost every $t>0$. Hence, $\phi(t,\cdot)p_t$ must be constant almost everywhere  in $x$. Since it is integrable over $\sR$, this constant must be zero, contradicting $p_t>0$ and $|\phi|=1$. Thus, no strict Markov policy can induce $\boldsymbol{\bar \mu} $.
\end{proof}

\begin{proof}[Proof of Example \ref{ex:nonsmooth_value}]
  We   verify that 
 $  V$ 
is a viscosity solution of \eqref{eq:tracking_hjb}.
By definition, $V(T,\cdot)$  satisfies
the terminal condition.
Now let 
$z_0\coloneqq (t_0,x_0)\in [0,T)\times \sR$.
If $  |x_0-m^a|<T-t_0$, then $V $ vanishes in a neighborhood of $z_0$ and
solves \eqref{eq:tracking_hjb} classically.
If $\ell\coloneqq |x_0-m^a|-(T-t_0)>0$,  then 
$\partial_t V(z_0) =|D_xV (z_0) |=e^{-\ell}$ and $D^2_xV(z_0) =-e^{-\ell}$.
Since $(-1+\tfrac12e^{-\ell})_+=0$,   \eqref{eq:tracking_hjb} also holds classically  at $z_0$. 

It remains to check the   interfaces
$x_0=m^a\pm(T-t_0)$ for $t_0\in [0,T)$.
Note that $V(t_0,\cdot)$ is not 
differentiable at these interfaces and  $D_xV(t_0,\cdot)$ has an upward jump:
from $-1$ to $0$
at the left interface, and from $0$ to $1$ at the right one.
Consequently, no $C^{1,2}$ test function can touch $V$ from above
at such a point, so the   subsolution condition holds
automatically.
For the supersolution condition, 
suppose  that $\zeta\in C^{1,2}$ touches $V$ from below
at an interface point $z_0=(t_0,x_0)$.
At the left interface, the  one-sided spatial slopes  give
$D_x\zeta (z_0)\in[-1,0]$.
Along the curve $x=m^a-(T-t)$, $V=0$ and 
 $t\mapsto \zeta(t,m^a-T+t)$ is maximized at $t_0$. Hence 
 $\partial_t\zeta (z_0)+ D_x\zeta  (z_0) \le 0$,
 with equality if $t_0>0$.
 Similarly, 
at the right interface, 
$D_x\zeta (z_0)\in[0,1]$,
and $\partial_t\zeta (z_0)- D_x\zeta  (z_0) \le 0$.
Thus in both cases, $\partial_t\zeta (z_0) \le  | D_x\zeta  (z_0)|$.
 This along with  
$( 1+\tfrac12 D^2_x\zeta (z_0))_-\ge0$ proves the supersolution condition.
This finishes the proof. 
\end{proof}
\section{Conclusion}
{This paper establishes a primal-dual system for continuous-time MFGs and for characterizing the set of all MFEs.
This system  requires neither the  convexity of   the associated Hamiltonian nor the  uniqueness of its optimizer, and remains applicable  even when
 the HJB equation has no classical solutions.

 This primal-dual formulation provides new analytical tools for studying MFEs. Theorem~\ref{thm:primal_dual_NE_diffusion_sufficient} gives a verification theorem under measurable coefficients, {suggesting a new route to constructing MFEs beyond the usual continuity assumptions \cite{lacker2015mean}.}
Meanwhile, this primal-dual system lays a new foundation for developing numerical methods to compute multiple MFEs. Indeed, for discrete-time MFGs \cite{guo2024mf} proposed efficient optimization algorithms  that identify equilibria by jointly searching for primal and dual variables which minimize the violation of the primal-dual system;
see  also  \cite{briceno2019implementation} for primal-dual-based algorithms for MFGs with local coupling, and \cite{achdou2016mean, achdou2016meanII} for mean field control problems}. Extending such an optimization-based approach to continuous-time MFGs 
is an interesting topic.

\bibliographystyle{siam}
\bibliography{LP.bib}

 \end{document}